\documentclass[12pt]{article}
\usepackage{mathrsfs}
\usepackage{}
\usepackage{graphicx}
\usepackage{amsmath}
\usepackage{amsfonts}
\usepackage{amssymb}
\usepackage{textcomp}
\usepackage{latexsym}
\usepackage{color}
 \makeatletter
\def\EquationsBySection{\def\theequation
{\thesection.\arabic{equation}}%
\@addtoreset{equation}{section}}

\makeatother \EquationsBySection

\font\sevenmsx=msam7 \font\fivemsx=msam5 \font\tenmsy=msbm10
\font\sevenmsy=msbm7 \font\fivemsy=msbm5
\newfam\msxfam
\newfam\msyfam
\textfont\msxfam=\tenmsx  \scriptfont\msxfam=\sevenmsx
\scriptscriptfont\msxfam=\fivemsx \textfont\msyfam=\tenmsy
\scriptfont\msyfam=\sevenmsy \scriptscriptfont\msyfam=\fivemsy
\def\Bbb{\ifmmode\let\next\Bbb@\else
\def\next{\errmessage{Use \string\Bbb\space only in math mode}}\fi\next}
\def\Bbb@#1{{\Bbb@@{#1}}}
\def\Bbb@@#1{\fam\msyfam#1}

\newtheorem{theorem}{\bf Theorem}[section]
\newtheorem{lemma}[theorem]{\bf Lemma}
\newtheorem{proposition}[theorem]{\bf Proposition}

\newenvironment{proof}{Proof:}{\quad \hfill $\Box$\vspace{2ex}}

\newcommand{\refts}{\hat{\mathcal{T}}^*}

\newcommand{\refpsi}{\hat{\psi}}

\newcommand{\isqrt}[1]{{\left(#1\right)^{1/2}}}

\usepackage{latexsym,amssymb,amsmath}
\begin{document}
\title{\bf Nonconforming Finite Volume Methods for Second Order Elliptic Boundary Value Problems
}

\author{Yuanyuan Zhang\thanks{Corresponding author. Department of Mathematics and Information Science, Yantai University, Yantai
264005, P. R. China. E-mail address: yyzhang@ytu.edu.cn. Supported in part by the National Natural
Science Foundation of China under grant 11426193, by Shandong Province Natural Science Foundation under grant ZR2014AP003 and by Shandong Province Higher Educational Science and Technology Program under grant J14li07.
} \ and
Zhongying Chen\thanks{Guangdong Province Key Laboratory of Computational Science, School of Mathematics and Computational Sciences, Sun Yat-sen University, Guangzhou 510275, P. R. China. E-mail addresses: lnsczy@mail.sysu.edu.cn. Supported in part by the
Natural Science Foundation of China under grants 10771224 and
11071264.} }

\date{}
\maketitle

\begin{abstract}
This paper is devoted to analyze of nonconforming finite volume methods (FVMs), whose trial spaces are chosen as the nonconforming finite element (FE) spaces, for solving the second order elliptic boundary value problems. We formulate the nonconforming FVMs as special types of Petrov-Galerkin  methods and develop a general convergence theorem,
which serves as a guide for the analysis of the nonconforming FVMs.
%We shall show that the regularity of the primary meshes, two discrete norm inequalities together with the uniform ellipticity of the family of the discrete bilinear forms lead to that the nonconforming FVM equation has a unique solution which enjoys the same optimal error order as the nonconforming FE methods.
As special examples, we shall present the triangulation based Crouzeix-Raviart (C-R) FVM  as well as the rectangle mesh based hybrid Wilson FVM. Their optimal error estimates in the mesh dependent $H^1$-norm will be obtained under the condition that the primary mesh is regular.
For the hybrid Wilson FVM, we prove that it enjoys the same optimal error order in the $L^2$-norm as that of the Wilson FEM.
Numerical experiments are also presented to confirm the theoretical results.
\end{abstract}


\vspace{2mm}{\bf AMS 2000 subject classifications.} 65N30, 65N12

%%%%%%%%%%%%%%%%%%%%%%%%%%%%%%%%%%%%%%%%%%%%%%%%%%%%%%%%%%%%%%%%%%%%%%%%%%%
\section{Introduction}
%%%%%%%%%%%%%%%%%%%%%%%%%%%%%%%%%%%%%%%%%%%%%%%%%%%%%%%%%%%%%%%%%%%%%%%%%%%
%The FVM has become one of the major numerical methods for solving partial differential equations in the past several decades.
Preserving certain local conservation laws and flexible algorithm constructions are the most attractive advantages of the FVM.  Due to its strengths,
the FVM has been widely used in
numerical solutions of PDEs, especially in computational fluid
dynamics, computational mechanics and hyperbolic problems
(cf. \cite{Em, LiJian, VM}).
In the past several decades, many researchers have studied
this method extensively and obtained some important results. We refer to \cite{BG, Cai, Cl, CWX, H, LCW, XZ} for an incomplete list of references.

Most of the existing work about FVMs for solving the second order elliptic boundary value problems focuses on the conforming schemes, which employ the standard conforming FE spaces as their trial spaces, see \cite{Bank, CXZ, ELL, F, L3} for triangulation based FVMs and \cite{LL2, Sch, ZZ} for rectangle mesh based FVMs.   There are little work about the nonconforming FVMs (cf. \cite{BR, CP1, CP2, ChouYe}).
A general construction of higher-order FVMs based on triangle meshes was
proposed in a recent paper \cite{CXZ} for solving the second order elliptic boundary problems and a unified
approach for analyzing the methods was developed.
We feel it is necessary to establish a unified theoretical framework for the nonconforming FVMs for solving boundary value problems of the two dimensional elliptic
equations.

In this paper, we shall establish a convergence theorem  applicable to the nonconforming triangle mesh based FVMs as well as the rectangle mesh based FVMs for solving the second order elliptic boundary problems.
We will see that comparing with the conforming FVMs, verifying the uniform boundedness and the uniform ellipticity of the family of the discrete bilinear forms
is still a task for the nonconforming FVMs. Moreover, there is an additional nonconforming error to estimate.
%In the general convergence theorem, we will show that two discrete norm inequalities leads to the uniform boundedness of the family of the discrete bilinear forms under the condition that the primary mesh is regular.

As a special example, the C-R FVM will be presented in this paper, whose trial space is the C-R FE space  with respect to the primary triangulation (cf. \cite{CR}) and test space is spanned by the characteristic functions of the control volumes in the dual partition. Based on the C-R element, paper
\cite{CP1} considered the FVM for solving elliptic boundary problems in 2-D and obtained the optimal order error estimates in
the $L^2$-norm and a mesh dependent $H^1$-norm. The the reaction term of the elliptic equation there was not generalized by the Petrov-Galerkin formulation. Instead,  this term was discretized using a diagonal matrix. By virtue of the same discretization skill of the reaction term, paper \cite{BR} considered the FVM based on the C-R element for the non-self-adjoint and indefinite elliptic
problems and proved the existence,
uniqueness and uniform convergence of the FV element approximations under minimal elliptic regularity assumption.
In the nonconforming FVM schemes presented in this paper, we employ the generalization of the Petrov-Galerkin formulation
to get the discrete bilinear forms. This will be beneficial to the development of a general framework for the numerical analysis of the methods.
%In this paper, the C-R FVM is formulated as a special type of the Petrov-Galerkin method.
We will prove
two discrete norm inequalities which lead to the uniform boundedness of the family of the discrete bilinear forms and we will establish the uniform ellipticity of the family of the discrete bilinear forms. We also show that the nonconforming error is equal to zero and in turn get the optimal error estimate in the mesh dependent $H^1$-norm for the C-R FVM.

Another special example, the hybrid Wilson FVM, will also be presented in this paper. The trial space of the hybrid Wilson FVM is the Wilson FE space with respect to the primary rectangle mesh and test space is panned by the characteristic functions of the control volumes combined with certain linearly independent
functions of the trial spaces. The hybrid FVM was initially constructed for a triangulation based quadratic FVM in \cite{Cl} and further studied in \cite{CXZ}.
We will show that the convergence order of the hybrid Wilson FVM in the mesh dependent $H^1$-norm is $O(h)$, the same as that for the Wilson FE method (cf. \cite{Sh}). The discrete bilinear form of the FVM is dependent on the meshes which introduce a major obstacle for the $L^2$-norm error estimate of the hybrid Wilson FVM. We note that the test space of the hybrid Wilson FVM is produced by the piecewise constant functions with respect to the dual partition  and the nonconforming functions of the trial space. Then, we may borrow some useful techniques used for the $L^2$-error estimate of the lower-order FVM (\cite{LL2}) and the Wilson FEM (\cite{Sh}). We will verify that the convergence order of the hybrid Wilson FVM in the $L^2$-norm is $O(h^2)$, the same as that for the Wilson FE method.

The remainder of this paper is organized as follows. In section 2, we describe the framework of the nonconforming FVMs for the second order elliptic
boundary value problems and develop a convergence theorem. Sections 3 and 4 are devoted to the discussion of the C-R FVM and the hybrid Wilson FVM respectively. Their discrete norm inequalities will be proved, nonconforming error term will be estimated and uniform ellipticity will be established. Then, their the optimal error estimates in the mesh dependent $H^1$-norm are derived, respectively. In section 5, we discuss the $L^2$-norm error estimate for the hybrid Wilson FVM for solving the Poisson equation. In the last section, we present a numerical example to confirm the convergence results in this paper.

In this paper, the notations of Sobolev spaces and associated norms are the same as those in \cite{Ciarlet} and $C$ will denote a generic positive constant independent of meshes and may be different at different occurrences.

%%%%%%%%%%%%%%%%%%%%%%%%%%%%%%%%%%%%%%%%%%%%%%%%%%%%%%%%%%%%%%%%%%%%%%%%%%%
\section{The Nonconforming FVMs for Elliptic Equations}\label{sec_schemes}
%%%%%%%%%%%%%%%%%%%%%%%%%%%%%%%%%%%%%%%%%%%%%%%%%%%%%%%%%%%%%%%%%%%%%%%%%%%

Let $\Omega$ be a polygonal domain in $\mathbb{R}^2$ with boundary
$\partial \Omega$. Suppose that  ${\bf a}:=[a_{ij}(x)]$ is a
$2\times 2$ symmetric matrix of functions $a_{ij}\in W^{1,
\infty}(\Omega)$ and $f\in
L^2(\Omega)$ and that $b$ is a smooth, nonnegative and real function. We consider the Dirichlet problem of the second order
 partial differential equation
\begin{equation}\label{eq:poisson_equation}
\left\{
\begin{array}{ll}
-\nabla\cdot (\mathbf{a}\nabla u)+bu=  f, &  \quad  \textrm{in } \Omega, \\
u =  0, & \quad  \textrm{on } \partial \Omega,
\end{array}
\right.
\end{equation}
where  $u$ is the unknown to be determined.
We assume that the
coefficients in equation \eqref{eq:poisson_equation} satisfy the
elliptic condition
$\sum\limits_{i,j=1}^2 a_{ij}(x)\xi_i\xi_j \ge r \sum\limits_{j=1}^2
\xi_j^2$,  for some $r>0$,  for all
$(\xi_i,\xi_j)\in \mathbb{R}^2$ and  all  $x\in
\Omega.$
%We assume that there exists a positive constant $a_0$ such that
%\begin{equation}\label{elliptic0}
%\xi^T \mathbf{a(x)}\xi \geq a_0 \xi^T \xi, \quad b(x) \geq 0, \quad \forall \xi\in \mathbb{R}^2, \forall x\in \overline{\Omega}.
%\end{equation}

Let $\mathcal{T}:=\{K\}$ be a partition of ${\Omega}$
in the sense that different elements in $\mathcal {T}$ have no overlapping interior,
vertices of $K$ in $\mathcal {T}$ do not belong to the
interior of an edge of any other elements in $\mathcal {T}$ and
$
\bar{\Omega}=\bigcup_{K\in \mathcal{T}}K
$. It may be a triangulation or a rectangle partition. Let $\mathcal {T}^*$ be another partition of $\Omega$ associated with $\mathcal {T}$, which is called a dual partition of $\mathcal {T}$.
Associated with $\mathcal {T}$ and $\mathcal {T}^*$, we define respectively the space
$$
\mathbb{H}^2_{\mathcal{T}}(\Omega):=\{v: v\in L^2(\Omega), v|_{K}\in \mathbb{H}^2(K),\
{\rm for \ all}\ K\in \mathcal{T}\}
$$
and the space
$$
\mathbb{H}^1_{\mathcal{T}^{*}}(\Omega):=\{v: v\in L^2(\Omega), \ v|_{K^*}\in
\mathbb{H}^1(K^*),\ {\rm for \ all}\ K^*\in \mathcal{T}^{*},\  \textrm{and}\  v|_{\partial\Omega} = 0\}.
$$
We introduce the discrete bilinear form for $w\in
\mathbb{H}^2_{\mathcal{T}}(\Omega)$ and $v\in \mathbb{H}^1_{\mathcal{T}^{*}}(\Omega)$ by setting
%$$
%a_{\mathcal{T}}(w,v) := \sum\limits_{K^{*}\in
%\mathcal{T}^{*}} \left\{ \int_{K^{*}} \big(\nabla w^T \mathbf{a}
%\nabla v+bwv\big)  - \int_{\partial K^{*}\setminus
%\partial \Omega} v(\mathbf{a}\nabla w)\cdot \mathbf{n}
%\right\},
%$$
\begin{equation}\label{biliner_form}
a_{\mathcal{T}}(w,v) :=\sum_{K\in\mathcal{T}}a_K \left(w, v\right)
\end{equation}
where
$$
a_K \left(w, v\right):=\sum\limits_{K^{*}\in
\mathcal{T}^{*}} \left\{ \int_{K^{*}\cap K} \big(\nabla w^T \mathbf{a}
\nabla v+bwv\big)  - \int_{\partial K^{*}\cap \mathrm{int} K} v(\mathbf{a}\nabla w)\cdot \mathbf{n}
\right\}
$$
and
$\mathbf{n}$ is the outward unit normal vector on $\partial
K^{*}$.
Employing the Green formula on the dual elements, we can show for $w\in \mathbb{H}_0^1(\Omega) \cap
\mathbb{H}^2(\Omega)$ and $v\in \mathbb{H}^1_{\mathcal{T}^{*}}(\Omega)$ that
$$
    a_{\mathcal{T}}(w,v) = \int_{\Omega} \big(-\nabla\cdot (\mathbf{a}\nabla w)+bw\big)v.
$$
The variational form for \eqref{eq:poisson_equation} is written
as finding $u\in \mathbb{H}_0^1(\Omega)\cap \mathbb{H}^2_{\mathcal{T}}(\Omega)$ such
that
\begin{equation}\label{variationaleq}
a_{\mathcal{T}}(u, v) = (f, v), \ \ \mbox{for all} \
\ v\in \mathbb{H}^1_{\mathcal{T}^{*}}(\Omega).
\end{equation}

We introduce the nonconforming FVMs for solving \eqref{eq:poisson_equation}.
Choose the finite dimensional trial space $\mathbb{U}_\mathcal {T}\subset \mathbb{H}^2_{\mathcal{T}}(\Omega)$ as a standard nonconforming FE space with respect to $\mathcal {T}$.
We choose the finite dimensional test space $\mathbb{V}_{\mathcal{T}^*}$ such that $\dim \mathbb{V}_{\mathcal{T}^*}= \dim \mathbb{U}_\mathcal {T}$ and for all $K\in\mathcal{T}$ and all $K^*\in \mathcal{T}^{*}$, the functions in $\mathbb{V}_{\mathcal{T}^*}$ restricted on $K\cap K^*$ are polynomials and moreover, the characteristic functions of $K^*\in \mathcal{T}^{*}$ are contained in $\mathbb{V}_{\mathcal{T}^*}$.
The nonconforming FVM for solving (\ref{eq:poisson_equation}) is a finite-dimensional approximation scheme which finds
$u_\mathcal {T}\in \mathbb{U}_\mathcal {T}$ such that
\begin{equation}
\label{eq:NFVM_sheme} a_{\mathcal{T}}\left(u_\mathcal{T}, v\right) = \left(f, v\right), \ \
\mbox{for all} \ \ v\in \mathbb{V}_{\mathcal{T}^*}.
\end{equation}
%where
%$$
%a_{\mathcal{T},\mathcal {T}^*}\left(u_\mathcal{T}, v\right):= \sum\limits_{K\in
%\mathcal{T}}  \sum\limits_{K^{*}\in
%\mathcal{T}^{*}} \left\{ \int_{K^{*}\cap K} \big(\nabla w^T \mathbf{a}
%\nabla v+bwv\big)  - \int_{\partial K^{*}\cap \mathrm{int} K} v(\mathbf{a}\nabla w)\cdot \mathbf{n}
%\right\}=\sum\limits_{K\in
%\mathcal{T}}a_{\mathcal{T}^*}\left(u_\mathcal{T}|_K, v|_K\right) ,
%$$
%For $w\in \mathbb{H}_0^1(\Omega) \cap
%\mathbb{H}^2_{\mathcal{T}}(\Omega)$ and $v\in \mathbb{H}^1_{\mathcal{T}^{*}}(\Omega)$,
%$$
%a_{\mathcal{T}^*}(w,v)=a_{\mathcal{T},\mathcal {T}^*}(w,v).
%$$
%However, for the hybrid FVMs, since $\mathbb{V}_{\mathcal{T}^*}\nsubseteq  \mathbb{H}^1_{\mathcal{T}^{*}}(\Omega)$, there does not hold
%$$
%a_{\mathcal{T}^*}(u, v) = (f, v), \ \ \mbox{for all} \
%\ v\in \mathbb{V}_{\mathcal{T}^*}.
%$$

In the rest of this section, we will establish a convergence theorem which serves as a guide for the numerical analysis of the nonconforming FVMs.
To this end, we do some preparations.
Let $h_K$ and $|K|$ be the
diameter and area of $K\in \mathcal {T}$ respevtively.
Let $\mathscr{T}:=\{\mathcal {T}\}$ denote a family of partitions of $\Omega$. Let $h$ be the largest diameter of $K\in \cup_{\mathcal {T}\in \mathscr{T}}\mathcal{T}$.
We say that the family $\mathscr{T}$ of the primary partitions is \emph{regular} if there exists a positive constant $\varrho$ such that for all $\mathcal {T}\in \mathscr{T}$ and all $K\in \mathcal {T}$
\begin{equation}\label{regularity}
\varrho_K \geq \varrho h_K,
\end{equation}
where by $\varrho_K$ we denote the diameter of the largest circle contained in $K$.
For each $K\in \mathcal {T}$, let $L^*_K$ denotes the dual gridlines contained in $K$.

By the trace inequality and the regularity of $\mathscr{T}$, we can derive the following lemma.
\begin{lemma}\label{line_intergral}
If $\mathscr{T}$ is regular, then for all $\mathcal{T}\in \mathscr{T}$, all $K\in \mathcal{T}$ and all $\ell^*\in L^*_K$ and for all $w \in \mathbb{H}^2_{\mathcal{T}}(\Omega)$
$$
\int_{\ell^*}|\nabla w|^2 ds \leq C h_K^{-1} (|w|_{1,K}^2 + h_K^2 |w|_{2,K}^2).
$$
\end{lemma}

For each $w \in \mathbb{H}^2_{\mathcal{T}}(\Omega)$, we define the semi-norms
$$
\|w\|_{1,\mathcal{T}} := \left(\sum_{K\in\mathcal{T}} \left| w
\right|_{1,K}^2 \right)^{1/2},  \quad   |w|_{2,\mathcal{T}} := \left(\sum_{K\in\mathcal{T}} \left| w
\right|_{2,K}^2 \right)^{1/2}.
$$
Usually, $\|\cdot\|_{1,\mathcal{T}}$ is a norm on the trial space $\mathbb{U}_\mathcal {T}$.
We introduce a discrete norm on the test space.
For any $v\in
\mathbb{V}_{\mathcal {T}^*}$, define
\begin{equation}\label{discretenorm_test}
|v|_{1,\mathbb{V}_{\mathcal {T}^*},K}: = \isqrt{\sum_{K^*\in\mathcal{T}^*}
|v|_{1,K^*\cap K}^2 + \sum_{\ell^*\in
L^*_K}|\ell^*|^{-1}\int_{\ell^*} [v]^2 }\ \ \mbox{and} \ \
|v|_{1,\mathbb{V}_{\mathcal {T}^*}}: = \isqrt{\sum_{K \in \mathcal{T}}
|v|_{1,\mathbb{V}_{\mathcal {T}^*},K}^2}.
\end{equation}
We assume that for all $\mathcal {T}\in \mathscr{T}$ and the associated $\mathcal {T}^*$ there exists linear mappings $\Pi_{\mathcal {T}^*}: \mathbb{U}_\mathcal {T}\rightarrow \mathbb{V}_{\mathcal {T}^*}$ with $\Pi_{\mathcal {T}^*} \mathbb{U}_\mathcal {T} = \mathbb{V}_{\mathcal {T}^*}$ satisfying the conditions that
\begin{equation}\label{boundedness_dis_norm1}
|\Pi_{\mathcal {T}^*}v|_{1,\mathbb{V}_{\mathcal {T}^*}} \leq C \|v\|_{1,\mathcal{T}}, \quad  \mbox{for all}  \ v\in \mathbb{U}_\mathcal {T}
\end{equation}
and
\begin{equation}\label{boundedness_dis_norm2}
\|\Pi_{\mathcal {T}^*}v\|_{0,\Omega} \leq C \|v\|_{1,\mathcal{T}}, \quad  \mbox{for all}  \ v\in \mathbb{U}_\mathcal {T}
\end{equation}

\begin{lemma}\label{boundedness}
If $\mathscr{T}$ is regular and the assumptions (\ref{boundedness_dis_norm1}) and (\ref{boundedness_dis_norm2}) hold, then there exists a positive constant $\gamma$ such that
for all $\mathcal {T}\in\mathscr{T}$, and for all $w \in \mathbb{H}^2_{\mathcal{T}}(\Omega)$  and  $v\in
\mathbb{U}_{\mathcal{T}}$
\begin{equation}\label{boudedness_ineq}
|a_{\mathcal{T}}(w,\Pi_{\mathcal{T}^{*}}v)| \leq \gamma (\|w\|_{0,\Omega}+\|w\|_{1,\mathcal {T}}+h|w|_{2,\mathcal{T}})\|v\|_{1,\mathcal {T}}.
\end{equation}
\end{lemma}
\begin{proof}
For each $\mathcal {T}\in\mathscr{T}$ and its associated $\mathcal{T}^{*}$ and each $v\in \mathbb{U}_{\mathcal{T}}$, let $v^*:=
\Pi_{\mathcal{T}^{*}}v$.
We note that
\begin{equation}\label{boundedness_mid4}
a_{\mathcal{T}}\left(w, v^*\right) = a_{c,\mathcal{T}}\left(w, v^*\right) + a_{d,\mathcal{T}}\left(w, v^*\right).
\end{equation}
where
$$
a_{c,\mathcal{T}}\left(w, v^*\right) := \sum_{K\in\mathcal{T}}\sum_{K^*\in \mathcal {T}^*} \int_{K^*\cap K}\big(\nabla w^T \mathbf{a}
\nabla v^*+bwv^*\big), \quad a_{d,\mathcal{T}}\left(w, v^*\right) := - \sum_{K\in\mathcal{T}}\sum_{\ell^*\in L^*_K}\int_{\ell^*}
[v^*](\mathbf{a}\nabla w)\cdot \mathbf{n}.
$$

We first estimate $a_{c,\mathcal{T}}\left(w, v^*\right)$.
By virtue of the Cauchy-Schwartz inequality, there holds
\begin{equation}\label{boundedness_mid1}
|a_{c,\mathcal{T}}\left(w, v^*\right)|\leq \|\mathbf{a}\|_{\infty} \|w\|_{1,\mathcal {T}} |v^*|_{1,\mathbb{V}_{\mathcal {T}^*}}+
 \left\|b\right\|_{L^{\infty}(\Omega)} \|w\|_{0,\Omega} \|v^*\|_{0,\Omega}.
\end{equation}
Combining (\ref{boundedness_mid1}) with the assumptions (\ref{boundedness_dis_norm1}) and (\ref{boundedness_dis_norm2}) yields
\begin{equation}\label{boundedness_mid6}
|a_{c,\mathcal{T}}\left(w, v^*\right)|\leq C  (\|w\|_{0,\Omega}+\|w\|_{1,\mathcal {T}})\|v\|_{1,\mathcal {T}}.
\end{equation}

We next estimate $a_{d,\mathcal{T}}\left(w, v^*\right)$.
Application of the Cauchy-Schwartz inequality gives that
\begin{equation}\label{boundedness_mid7}
|a_{d,\mathcal{T}}\left(w, v^*\right)|\leq |v^*|_{1,\mathbb{V}_{\mathcal {T}^*}}
\cdot
\left(\sum_{K\in\mathcal{T}}\sum_{\ell^*\in L_K^*} |\ell^*| \int_{\ell^*}
\big((\mathbf{a}\nabla w)\cdot \mathbf{n}\big)^2 ds\right)^{1/2}.
\end{equation}
Since $\mathscr{T}$ is regular, it follows from Lemma \ref{line_intergral} that
\begin{equation}\label{boundedness_mid2}
|\ell^*| \int_{\ell^*}
\big((\mathbf{a}\nabla w)\cdot \mathbf{n}\big)^2 ds \leq \|\mathbf{a}\|_{\infty}^2 h_K \int_{\ell^*} |\nabla w|^2 ds \leq C\|\mathbf{a}\|_{\infty}^2 (|w|_{1,K}^2 + h_K^2 |w|_{2,K}^2).
\end{equation}
Substituting (\ref{boundedness_mid2}) and the assumption (\ref{boundedness_dis_norm1}) into (\ref{boundedness_mid7}), we obtain
\begin{equation}\label{boundedness_mid3}
|a_{d,\mathcal{T}}\left(w, v^*\right)|\leq C
\left(\|w\|_{1,\mathcal {T}}+h|w|_{2,\mathcal{T}}\right) \|v\|_{1,\mathcal {T}}.
\end{equation}
Combining (\ref{boundedness_mid4}) with (\ref{boundedness_mid6}) and (\ref{boundedness_mid3}) yields the desired result of this lemma.
\end{proof}

If there exists a constant $\gamma>0$ independent of meshes such that
inequality (\ref{boudedness_ineq}) holds, we say that the family
$
\mathscr{A}_\mathscr{T}:=\{a_{\mathcal{T}}(\cdot,\Pi_{\mathcal {T}^*} \cdot):\mathcal {T}\in \mathscr{T}\}
$
of the discrete bilinear forms is {\em uniformly bounded}.  Lemma \ref{boundedness} shows that the regularity of $\mathscr{T}$ and the assumptions (\ref{boundedness_dis_norm1}) and (\ref{boundedness_dis_norm2}) are sufficient conditions for the uniform boundedness of $\mathscr{A}_\mathscr{T}$.
We furthermore assume that $\mathscr{A}_\mathscr{T}$
is {\em uniformly elliptic}, that is, there exists a constant $\sigma>0$ such that for all $\mathcal {T}\in \mathscr{T}$ and the associated $\mathcal {T}^*$
\begin{equation}\label{uniform_elliptic}
a_{\mathcal{T}}(w, \Pi_{\mathcal{T}^{*}}w)\ge \sigma
\|w\|_{1,\mathcal {T}}^2,
    \quad  \mbox{for all} \
w\in \mathbb{U}_\mathcal {T}.
\end{equation}
%If there exists a constant $\sigma>0$ such that (\ref{uniform_elliptic}) holds, we say that the family .

We present the convergence of the nonconforming FVMs.

\begin{theorem}\label{thm: convergence_theorem}
Let $u\in \mathbb{H}_0^1(\Omega)\cap \mathbb{H}^2(\Omega)$ be the solution of
(\ref{eq:poisson_equation}). If $\mathscr{T}$ is regular and the assumptions (\ref{boundedness_dis_norm1}), (\ref{boundedness_dis_norm2}) and (\ref{uniform_elliptic}) hold,
then for each $\mathcal {T}\in\mathscr{T}$ the FVM equation (\ref{eq:NFVM_sheme}) has a unique solution $u_\mathcal{T}\in \mathbb{U}_\mathcal{T}$, and there exists a
positive constant $C$ such that for all $\mathcal{T}\in\mathscr{T}$
\begin{equation}\label{thm: convergence_theorem_ineq}
\|u-u_\mathcal{T}\|_{1,\mathcal {T}}\le C\left(\inf_{w\in \mathbb{U}_\mathcal{T}} \left(\|u-w\|_{0,\Omega}+ \|u-w\|_{1,\mathcal {T}} + h|u-w|_{2,\mathcal {T}} \right) + \sup_{v\in \mathbb{U}_\mathcal{T}} \frac{E_\mathcal {T}(u, v)}{|v|_{1,\mathcal {T}}}\right).
\end{equation}
where
\begin{equation}\label{nonconforming error}
E_\mathcal {T}(u,v):= a_{\mathcal {T}}(u-u_\mathcal {T}, \Pi_{\mathcal {T}^*}v) = a_{\mathcal {T}}(u, \Pi_{\mathcal {T}^*}v)- (f, \Pi_{\mathcal {T}^*}v).
\end{equation}
\end{theorem}
\begin{proof}
Assume that (\ref{eq:NFVM_sheme}) with $f=0$ has a nonzero solution
$u_\mathcal{T}\in \mathbb{U}_\mathcal{T}$.
From (\ref{uniform_elliptic}), we get that
$$
0=a_{\mathcal{T}}(u_\mathcal{T},
\Pi_{\mathcal{T}^{*}}u_\mathcal{T})\ge \sigma \|u_\mathcal{T}\|_{1,\mathcal {T}}^2
\neq 0.
$$
This contradiction ensures that the linear system resulting from (\ref{eq:NFVM_sheme}) has a unique solution.

For all $w\in
\mathbb{U}_\mathcal{T}$,
\begin{equation}\label{thm: convergence_theorem_mid1}
\|u- u_\mathcal{T}\|_{1,\mathcal {T}}\le \|u- w\|_{1,\mathcal {T}}+ \|w- u_\mathcal{T}\|_{1,\mathcal {T}}.
\end{equation}
Condition (\ref{uniform_elliptic}) ensures that
$$
\sigma \|w- u_\mathcal{T}\|_{1,\mathcal {T}}^2 \le a_{\mathcal{T}}(w- u_\mathcal{T},
\Pi_{\mathcal{T}^{*}}(w- u_\mathcal{T})) =a_{\mathcal{T}}(w- u,
\Pi_{\mathcal{T}^{*}}(w- u_\mathcal{T}))+ a_{\mathcal{T}}(u- u_\mathcal {T},
\Pi_{\mathcal{T}^{*}}(w- u_\mathcal{T})),
$$
which implies
$$
\|w- u_\mathcal{T}\|_{1,\mathcal {T}}\le \sigma^{-1} \sup_{v\in
\mathbb{U}_\mathcal{T}}\frac{a_{\mathcal{T}}(w- u,
\Pi_{\mathcal{T}^{*}}v)}{\|v\|_{1,\mathcal {T}}} + \sigma^{-1} \sup_{v\in \mathbb{U}_\mathcal{T}} \frac{E_\mathcal {T}(u, v)}{\|v\|_{1,\mathcal {T}}}.
$$
Since $\mathscr{T}$ is regular and (\ref{boundedness_dis_norm1}) and (\ref{boundedness_dis_norm2}) hold, by Lemma \ref{boundedness}
we observe that
\begin{equation}\label{thm: convergence_theorem_mid2}
\|w- u_\mathcal{T}\|_{1,\mathcal {T}}\le \sigma^{-1}\gamma
\left(\|u-w\|_{0,\Omega}+ \|u-w\|_{1,\mathcal {T}} + h|u-w|_{2,\mathcal {T}} \right) + \sigma^{-1} \sup_{v\in \mathbb{U}_\mathcal{T}} \frac{E_\mathcal {T}(u, v)}{\|v\|_{1,\mathcal {T}}}.
\end{equation}
From (\ref{thm: convergence_theorem_mid1}) and (\ref{thm: convergence_theorem_mid2}), we conclude that the desired inequality
(\ref{thm: convergence_theorem_ineq}) holds
with $C:=\max\{1+\sigma^{-1}\gamma, \sigma^{-1}\}$.
\end{proof}

%The error estimate for the nonconforming FVMs presented in this theorem can be regarded as an extension of that for the conforming FVMs.
Comparing with the error estimate inequality of the conforming FVMs (cf. Theorem 4.2 of \cite{CXZ}), the error estimate inequality in Theorem \ref{thm: convergence_theorem} for the nonconforming FVMs has one term (\ref{nonconforming error}) more, which is called \emph{the nonconforming error} term. This term is produced by the nonconforming character of the trial spaces.
Since the discrete bilinear forms of the nonconforming FVMs are dependent on the grids, similar as the conforming FVMs, verifying the the uniform ellipticity of the family of the discrete bilinear forms is still a task for the nonconforming FVMs.

In the following two sections, we shall present and analyze two specific nonconforming FVM schemes for solving the equation (\ref{eq:poisson_equation}) respectively.

\section{The C-R FVM}

%The C-R FVM employs the classical C-R finite element space associated with the primary triangulation $\mathcal {T}$ of $\Omega$ as its trial space and piecewise constant function space associated with the dual partition $\mathcal {T}^*$ as its test space.
In this section, we first present the scheme of the C-R FVM. We then verify the discrete norm inequalities (\ref{boundedness_dis_norm1}) and (\ref{boundedness_dis_norm2}), establish the uniform ellipticity of the family of the discrete bilinear forms and discuss the nonconforming error term. In turn, the optimal error estimate of the C-R FVM is obtained according to Theorem \ref{thm: convergence_theorem}.

%\subsection{The C-R FVM Scheme}
In the C-R FVM, the partition $\mathcal {T}$ is a triangulation of $\Omega$.
Any vertex of $\Omega$ is a vertex of a triangle in $\mathcal {T}$. We denote by $\mathcal {N}_{\mathcal {T}}$, $\mathcal {M}_{\mathcal {T}}$ and $\mathcal {Q}_{\mathcal {T}}$, respectively, the sets of vertices, midpoints of the edges
and barycenters of the triangles in $\mathcal {T}$. Let $\dot{\mathcal {N}}_{\mathcal {T}}:=\mathcal {N}_{\mathcal {T}}\setminus \partial\Omega$ and $\dot{\mathcal {M}}_{\mathcal {T}}:=\mathcal {M}_{\mathcal {T}}\setminus \partial\Omega$ be the set of interior vertices and interior midpoints, respectively.

The trial space $\mathbb{U}_\mathcal {T}$ of the C-R FVM is chosen as the classical C-R nonconforming finite element space, that is,
$$
\mathbb{U}_{\mathcal {T}}:= \{w\in L^2(\Omega): w\ \text{is linear on all}\ K\in \mathcal {T}, w\ \text{is continuous at}\ \dot{\mathcal {M}}_{\mathcal {T}}, w=0 \ \text{at}\ \mathcal {M}_{\mathcal {T}}\cap \partial\Omega \}.
$$
%The set of degrees of freedom $\Sigma := \{\eta \}$ is a set of point evaluation functionals at the points in $\dot{\mathcal {M}}_{\mathcal {T}}$.
%Suppose that the cardinal of the set $\dot{\mathcal {M}}_{\mathcal {T}}$ is $n$, then $\dim \mathbb{U}_\mathcal {T} =n$ and the cardinal of the set $\Sigma$ is $n$. Define a function
%\begin{equation}\label{Delta}
% \delta_{i,j} := \left\{
%\begin{array}{ll} 1, & i=j,\\0, &i\neq j.
%\end{array}
%\right.
%\end{equation}
%For a positive integer $m$, let $\mathbb{N}_{m}:=\{1,2,\ldots,m\}$.
%There is a basis $\Phi_\mathcal {T}:=\{\phi_i: i\in \mathbb{N}_n\}$ for $\mathbb{U}_\mathcal {T}$ such that
%\begin{equation}\label{freedomK_trial}
%    \eta_j(\phi_i) = \delta_{i,j}, \ \quad i,j \in \mathbb{N}_{n}.
%\end{equation}
Obviously, $\mathbb{U}_\mathcal {T}$ is not in the space $\mathbb{H}_0^1(\Omega)$. The C-R FVM is a kind of nonconforming FVM.

We describe the dual partition $\mathcal {T}^*$ and the test space $\mathbb{V}_{\mathcal {T}^*}$.
For each $M\in \dot{\mathcal {M}}_{\mathcal {T}}$, suppose that it is on an edge denoted by $P_iP_j$ and that $P_iP_j$ is a common edge of the triangles $\Delta P_iP_jP_k$ and $\Delta P_jP_iP_k'$ in $\mathcal {T}$. Let $Q$ and $Q'$ be the barycenters of $\Delta P_iP_jP_k$ and $\Delta P_iP_jP_k'$ respectively. We connect the points $P_i$, $Q$, $P_j$, $Q'$ and $P_i$ consecutively to derive a quadrilateral $K^*_M$ surrounding the point $M$ (cf. Figure \ref{fig:dual_partition}). For each $M\in {\mathcal {M}}_{\mathcal {T}}\backslash \dot{\mathcal {M}}_{\mathcal {T}} $, following the same process, we derive a triangle $K^*_M$ associated the point $M$. Let $\mathcal {T}^*:=\{K^*_M: M\in \mathcal {M}_{\mathcal {T}}\}$. The elements in $\mathcal {T}^*$ are called control volumes.
The test space $\mathbb{V}_{\mathcal {T}^*}$ is defined as follows
$$
\mathbb{V}_{\mathcal {T}^*}:= \{v\in L^2(\Omega): w|_{K^*}= \textrm{constant}, \mbox{for all}\ K^*\in \mathcal {T}^*,  v|_{\partial\Omega}=0 \}.
$$
%Obviously, the dimension of $\mathbb{V}_{\mathcal {T}^*}$ is equal to the cardinal of the set $\dot{\mathcal {M}}_{\mathcal {T}}$, that is, $\dim \mathbb{V}_{\mathcal {T}^*} =n$. The characteristic functions of $K^*_M, M\in \dot{\mathcal {M}}_{\mathcal {T}}$ form a basis for $\mathbb{V}_{\mathcal {T}^*}$. We denote by $\Psi_\mathcal {T}:=\{\psi_i: i\in \mathbb{N}_n\}$ this basis by properly sequencing the elements such that
%\begin{equation}\label{freedomK_test}
%    \eta_j(\psi_i) = \delta_{i,j}, \ \quad i,j \in \mathbb{N}_{n},
%\end{equation}
We note that $\mathbb{V}_{\mathcal {T}^*} \subseteq \mathbb{H}^1_{\mathcal{T}^{*}}(\Omega)$.

We use $\chi_{E}$ to denote the characteristic function of $E\subset\mathbb{R}^2$.
We define the invertible linear mapping
$\Pi_{\mathcal{T}^{*}}:\mathbb{U}_{\mathcal{T}} \to
\mathbb{V}_{\mathcal{T}^{*}}$ for any $w\in \mathbb{U}_{\mathcal{T}}$ by
$$
    \Pi_{\mathcal{T}^{*}}w:=\sum_{M\in \mathcal {M}_{\mathcal {T}}} w(M)\chi_{K^*_M}.
$$
Obviously, for each $w\in \mathbb{U}_{\mathcal{T}}$ and $K^*_M\in \mathcal {T}^*$, the restriction of $\Pi_{\mathcal{T}^{*}}w$ on $K^*_M$ is the constant function $w(M)$.

\begin{figure}[ht!]
\centering
\includegraphics[width=0.5\textwidth]{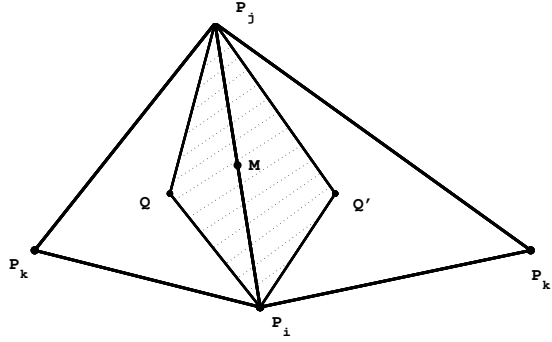}
\caption{The dual partition of the C-R FVM} \label{fig:dual_partition}
\end{figure}

For $w\in \mathbb{H}^2_{\mathcal{T}}(\Omega)$ and $v\in \mathbb{V}_{\mathcal {T}^*}$, from (\ref{biliner_form}), we derive the discrete bilinear form of the C-R FVM
%$$
%a_{\mathcal {T}}(w,v):= \sum_{K\in \mathcal {T}}a_K(w,v),
%$$
%where
\begin{equation}\label{bilinear_form_CR}
  a_K(w,v)= \sum\limits_{K^*\in\mathcal{T}^*} \left(\int_{K^*\cap K} bwv  -\int_{\partial
K^* \cap \mathrm{int} K}
 v(\mathbf{a}\nabla w)\cdot \mathbf{n}  \right),  \quad a_{\mathcal {T}}(w,v)= \sum_{K\in \mathcal {T}}a_K(w,v).
\end{equation}

\emph{Remark: }
In the FVM proposed in Paper \cite{CP1} for solving second order elliptic boundary value problems which is also based on the C-R element ,
the term $bu$ in (\ref{eq:poisson_equation}) is discretized using a diagonal matrix, that is, using $\Pi_{\mathcal{T}^{*}}w$ instead of $w$ in the term $\int_{K^*\cap K} bwv$ in (\ref{bilinear_form_CR}).
This processing may be viewed as producing an approximation of the discrete bilinear form given in (\ref{bilinear_form_CR}) and the theoretical framework given in Section \ref{sec_schemes} of this paper may cover the FVM scheme given in \cite{CP1}.

%In the framework of the FVM schemes presented in this paper, we employ the generalization of the Petrov-Galerkin formulation to get the discrete bilinear form without any special treatment of the term $bu$.

%In the following of this section, we devote to establishing the unform boundedness condition (C-1) and the unform local-ellipticity condition (C-2) for the C-R FVM. Then, according to Theorem \ref{thm: convergence_theorem}, the convergence result will be derived.

The remainder of this section is devoted to
the convergence analysis of the C-R FVM. According to Theorem \ref{thm: convergence_theorem}, we need to very conditions (\ref{boundedness_dis_norm1}), (\ref{boundedness_dis_norm2}) and (\ref{uniform_elliptic}) for the C-R FVM.
Given a $K\in \mathcal {T}$,
we denote the set of the sides of $K$ by $E(K)$ and let $m_{e}$ denote the midpoint of a side $e\in \bigcup_{K\in \mathcal {T}} E(K)$.
Note that for the C-R FVM the discrete norm for the test space defined in (\ref{discretenorm_test}) becomes
$$
|v|_{1,\mathbb{V}_{\mathcal {T}^*}} = \isqrt{\sum_{K \in \mathcal{T}} \sum_{e,l\in E(K)} \left(v(m_e)-v(m_l)\right)^2}.
$$

From Lemma 3.5 of \cite{CP1} and the definition of $\Pi_{\mathcal{T}^{*}}$, we derive that the norms $|\Pi_{\mathcal{T}^{*}}
\cdot|_{1,\mathbb{V}_{\mathcal {T}^*}}$ and $\|\cdot\|_{1,\mathcal {T}}$ are equivalent which implies (\ref{boundedness_dis_norm1}) for the C-R FVM.
\begin{lemma}\label{dnorm_equi}
There exist
positive constants ${c_1}$ and $c_2$ such that for all $\mathcal{T}\in\mathscr{T}$
and all $v\in \mathbb{U}_{\mathcal{T}}$,
$$
    c_1 \|v\|_{1,\mathcal {T}}\le
|\Pi_{\mathcal{T}^{*}}v|_{1,\mathbb{V}_{\mathcal {T}^*}} \le c_2
\|v\|_{1,\mathcal {T}}.
$$
\end{lemma}

The next lemma is given in Lemma 3.7 of \cite{CP1}.
\begin{lemma}\label{dis_Poin_inequ}
There exists a positive constant $C$ such that for all $\mathcal{T}\in\mathscr{T}$ and all $v\in \mathbb{U}_{\mathcal{T}}$
$$
\|v\|_{0,\Omega} \leq C \|v\|_{1,\mathcal {T}}.
$$
\end{lemma}

We choose the triangle $\hat{K}$ with vertices $\hat{P}_1:=(0,0)$, $\hat{P}_2:=(1,0)$ and $\hat{P}_3:=(0,1)$ as the reference triangle.
For any triangle $K$, there is an invertible affine mapping $\mathcal {F}_K$ from $\hat{K}$ to $K$ (cf. \cite{Ciarlet}).
\begin{lemma}\label{L2testspaceL2}
There exists a positive constant C such that for all $\mathcal{T}\in\mathscr{T}$ and all $v\in \mathbb{U}_{\mathcal{T}}$
$$
\|\Pi_{\mathcal{T}^{*}}v\|_{0,\Omega} \leq C \|v\|_{0,\Omega}.
$$
\end{lemma}
\begin{proof}
It suffices to prove that there exists a positive constant $C$ such that for each $\mathcal{T}\in\mathscr{T}$ and each $K\in \mathcal {T}$
\begin{equation}\label{L2testspace_mid1}
\|\Pi_{\mathcal{T}^{*}}v\|_{0,K}^2 \leq C \|v\|_{0,K}^2.
\end{equation}
From the definition of $\Pi_{\mathcal{T}^{*}}$, we get that
\begin{equation}\label{L2testspace_mid2}
\|\Pi_{\mathcal{T}^{*}}v\|_{0,K}^2 \leq |K| \sum_{e\in E(K)} v^2(m_{e}).
\end{equation}
By making use of the variable transformation from $K$ to the reference triangle $\hat{K}$, we derive that
%there exists a positive constant $C$ independent of mesh such that
$$
\|v\|_{0,K}^2 = 2|K|\int_{\hat{K}} |\hat{v}|^2
$$
Note that $\hat{v}= \sum_{e\in E(K)} v(m_{e}) \hat{\varphi}_e$, where $\hat{\varphi}_e$ are the basis of the trial space on $\hat{K}$. By simple calculation, we learn that the matrix $\mathbf{M}:=[\int_{\hat{K}} \hat{\varphi}_e \hat{\varphi}_l]$ is positive definite. Thus, there exists a positive constant $C$ independent of meshes such that
\begin{equation}\label{L2testspace_mid3}
\|v\|_{0,K}^2\geq C |K|   \sum_{e\in E(K)} v^2(m_{e})
\end{equation}
Combining (\ref{L2testspace_mid2}) and (\ref{L2testspace_mid3}) yields (\ref{L2testspace_mid1}).
\end{proof}

From Lemma \ref{dis_Poin_inequ} and   Lemma \ref{L2testspaceL2}, we immediately get inequality (\ref{boundedness_dis_norm2}) for the C-R FVM as presented in the following proposition.

\begin{proposition}\label{L2testspace}
There exists a positive constant $C$ such that for all $\mathcal{T}\in\mathscr{T}$ and all $v\in \mathbb{U}_{\mathcal{T}}$
$$
\|\Pi_{\mathcal{T}^{*}}v\|_0 \leq C \|v\|_{1,\mathcal {T}}.
$$
\end{proposition}

We study the uniform ellipticity condition (\ref{uniform_elliptic}) for the C-R FVM.
We will establish that when $h$ is sufficiently small, (\ref{uniform_elliptic}) holds.
For $w\in \mathbb{H}^2_{\mathcal{T}}(\Omega)$ and $v\in \mathbb{V}_{\mathcal {T}^*}$, let
$$
 a_{\mathcal {T},1}(w,v):= -\sum\limits_{K\in\mathcal {T}}\sum\limits_{K^*\in\mathcal{T}^*} \int_{\partial
K^* \cap \mathrm{int} K}
 v(\mathbf{a}\nabla w)\cdot \mathbf{n} ds \quad \text{and} \quad  a_{\mathcal {T},2}(w,v):= \sum\limits_{K\in\mathcal {T}} \int_{K} bwv
$$
Then
$$
 a_{\mathcal {T}}(w,v) = a_{\mathcal {T},1}(w,v) + a_{\mathcal {T},2}(w,v).
$$

The following lemma is derived from the  proof of Lemma 4.2 of \cite{CP1}.
\begin{lemma}\label{LocalEll1}
There exists a positive
constant $C$
such that for all $\mathcal{T}\in \mathscr{T}$ and its associated $\mathcal {T}^*$, all $w\in \mathbb{U}_\mathcal{T}$,
$$
a_{\mathcal {T},1}(w, \Pi_{\mathcal{T}^{*}}w)\ge C
\|w\|^2_{1,\mathcal {T}}.
$$
\end{lemma}

%We next analyze $a_{\mathcal {T},2}(\cdot,\cdot)$.
%To this end, we introduce some notations.
%Recall that $\hat{P}_i,i=1,2,3$ are the vertices of the reference triangle $\hat{K}$. Let $Q:=(1/3,1/3)$.
%We use $K_1^*$, $K_2^*$ and $K_3^*$ to denote the triangles with vertices $Q$, $\hat{P}_2$ and $\hat{P}_3$, $Q$, $\hat{P}_3$ and $\hat{P}_1$
%and $Q$, $\hat{P}_1$ and $\hat{P}_2$ respectively.
%Let $\chi_{K^*_i}, i\in \mathbb{N}_3$ be the character function of $K^*_i$.
In the next lemma, we estimate $a_{\mathcal {T},2}(\cdot, \Pi_{\mathcal{T}^{*}}\cdot)$.
\begin{lemma}\label{LocalEll2}
If the coefficient $b$ in (\ref{eq:poisson_equation}) is a piecewise constant function with respect $\mathcal{T}$, then for all $\mathcal{T}\in \mathscr{T}$ and its associated $\mathcal {T}^*$ and all $w\in \mathbb{U}_\mathcal{T}$,
$$
a_{\mathcal {T},2}(w, \Pi_{\mathcal{T}^{*}}w)\ge 0.
$$
Moreover, if and only if $b=0$ or $w=0$, the above inequality becomes an equality.
\end{lemma}
\begin{proof}
For all $w\in \mathbb{U}_\mathcal{T}$, let $w^*:=\Pi_{\mathcal{T}^{*}}w$.
%Note that $a_{\mathcal {T},2}(w,w^*)=\sum\limits_{K\in\mathcal {T}} \int_{K} bww^*$.
By changing variables, we derive that
\begin{equation}\label{LocalEll2_mid1}
\int_{K} ww^* =2|K|\int_{\hat{K}} \hat{w} \hat{w}^*.
\end{equation}
%Let
%$$
%\hat{\varphi}_1(x,y):= -1+ 2x+2y, \quad \hat{\varphi}_2(x,y):= 1-2x, \quad \hat{\varphi}_3(x,y):= 1-2y,  \quad  (x,y)^T\in\hat{K}.
%$$
%Note that $\hat{\varphi}_i, i=1,2,3$ is the basis of the trial space on $\hat{K}$. Then, there exist constant $c_i,i=1,2,3$ such that
%\begin{equation}\label{LocalEll2_mid2}
%\hat{w}:=\sum_{i=1}^3 c_i \hat{\varphi}_i \quad \text{and} \quad  \hat{w}^*:=\sum_{i=1}^3 c_i \chi_{K_i^*}.
%\end{equation}
%Substituting (\ref{LocalEll2_mid2}) into (\ref{LocalEll2_mid1}) yields
%$$
%\int_{K} ww^* = 2|K|c^T \mathbf{E}c,
%$$
%where $c:=[c_1,c_2,c_3]^T$, $\mathbf{E}_1:= [\int_{\hat{K}}\hat{\varphi}_i \chi_{K_j^*}: i,j=1,2,3]$ and $\mathbf{E}:= (\mathbf{E}_1^T+ \mathbf{E}_1)/2$.
%By simple calculation, we derive that $\mathbf{E}$ is positive definite.
By simple calculation, we derive that $\int_{\hat{K}} \hat{w} \hat{w}^*$ is a positive definite quadratic form of $w(m_e),e\in E(K)$.
Thus, $\int_{K} ww^* \geq 0$ and if and only if $w|_K=0$, the inequality sign becomes equal sign. Since $b$ is piecewise constant with $\mathcal {T}$ and $b\geq 0$, we get that
$$
a_{\mathcal {T},2}(w,w^*)=\sum\limits_{K\in\mathcal {T}} \int_{K} bww^* \geq \sum\limits_{K\in\mathcal {T}} b \int_{K} ww^*.
$$
This yields the desired results of this lemma.
\end{proof}

%\begin{lemma}\label{localE_equal_CR}
%If $h$ is sufficiently small and the coefficient $b\in W^{1,
%\infty}(\Omega)$ in (\ref{eq:poisson_equation}), then
%\end{lemma}

From Lemma \ref{LocalEll1} and Lemma \ref{LocalEll2}, we can get the following proposition.
\begin{proposition}\label{suffi_conditions_HQ}
If $h$ is sufficiently small,
then $\mathscr{A}_\mathscr{T}$ is uniformly elliptic.
\end{proposition}
\begin{proof}
We need to prove that (\ref{uniform_elliptic}) holds with a positive constant independent of meshes.
%For for all $\mathcal {T}\in \mathscr{T}$ and all $w\in \mathbb{U}_\mathcal {T}$.
If $b=0$, from Lemma \ref{LocalEll1}, (\ref{uniform_elliptic}) holds. We next assume that $b\neq 0$.
For each $K\in \mathcal {T}$, let $Q_K$ denote its barycenter and let $\bar{b}_K:=b(Q_K)$.
For for all $\mathcal {T}\in \mathscr{T}$ and all $w\in \mathbb{U}_\mathcal {T}$ and $w\neq 0$, let $w^*:=\Pi_{\mathcal{T}^{*}}w$.
We define
$$
\bar{a}_{\mathcal {T},2}(w,w^*):= \sum\limits_{K\in\mathcal {T}} \int_{K} \bar{b}_K ww^*.
$$
By the smoothness of $b$, we have that
$$
\lim_{h\rightarrow 0} \left({a}_{\mathcal {T},2}(w,w^*)- \bar{a}_{\mathcal {T},2}(w,w^*) \right)=0.
$$
Thus, by Lemma \ref{LocalEll2}, we learn that when $h$ is sufficiently small, ${a}_{\mathcal {T},2}(w,w^*) >0$. This combined with Lemma \ref{LocalEll1} yields (\ref{uniform_elliptic}).
\end{proof}

We analyze the nonconforming error as defined in (\ref{nonconforming error}) for the C-R FVM in the next proposition.
\begin{proposition}\label{pro:CRnonconforming-error}
Let $u\in \mathbb{H}_0^1(\Omega)\cap \mathbb{H}^2(\Omega)$ be the solution of
(\ref{eq:poisson_equation}) and $u_\mathcal{T}\in
\mathbb{U}_\mathcal{T}$ be the solution of the C-R FVM equation. Then, for each $v\in \mathbb{U}_\mathcal{T}$, the nonconforming error $E_\mathcal {T}(u,v)$ is equal to zero.
\end{proposition}
\begin{proof}
Note that in the C-R FVM, $\mathbb{V}_{\mathcal{T}^*} \subseteq \mathbb{H}^1_{\mathcal{T}^{*}}(\Omega)$. From (\ref{variationaleq}), we get that for each $v\in \mathbb{U}_\mathcal{T}$
\begin{equation}\label{pro:CRnonconforming-errormid1}
a_\mathcal{T}\left(u, \Pi_{\mathcal {T}^*}v\right) = \left(f, \Pi_{\mathcal {T}^*}v\right).
\end{equation}
From (\ref{eq:NFVM_sheme}), we obtain that
\begin{equation}\label{pro:CRnonconforming-errormid2}
a_\mathcal{T}\left(u_\mathcal{T}, \Pi_{\mathcal {T}^*}v\right) = \left(f, \Pi_{\mathcal {T}^*}v\right).
\end{equation}
Combining (\ref{pro:CRnonconforming-errormid1}) and (\ref{pro:CRnonconforming-errormid2}) yields
$$
a_\mathcal{T}\left(u-u_\mathcal{T}, \Pi_{\mathcal {T}^*}v\right) =0,
$$
which means that
$
E_\mathcal {T}(u,v)=0.
$
\end{proof}

Now we are ready to present the convergence of the C-R FVM.
\begin{theorem}\label{thm:CRconvergence}
Let $u\in \mathbb{H}_0^1(\Omega)\cap \mathbb{H}^2(\Omega)$ be the solution of
(\ref{eq:poisson_equation}).
If $\mathscr{T}$ is regular and $h$ is sufficiently small, then for each $\mathcal{T}$ the C-R FVM equation
has a unique solution $u_\mathcal{T}\in
\mathbb{U}_\mathcal{T}$, and there exists a positive constant $C$
such that for all $\mathcal{T}\in\mathscr{T}$
\begin{equation}\label{CRerror-estimate-inequality}
\|u-u_\mathcal{T}\|_{1,\mathcal {T}}\le C h|u|_{2}.
\end{equation}
\end{theorem}
\begin{proof}
Combining Theorem \ref{thm: convergence_theorem} with Lemma \ref{dnorm_equi}, Propositions \ref{L2testspace}, \ref{suffi_conditions_HQ} and \ref{pro:CRnonconforming-error}, we get that for each $\mathcal {T}\in\mathscr{T}$ the C-R FVM equation has a unique solution $u_\mathcal{T}\in \mathbb{U}_\mathcal{T}$, and there exists a
positive constant $C$ such that for all $\mathcal{T}\in\mathscr{T}$
\begin{equation}\label{CRconvergence1}
\|u-u_\mathcal{T}\|_{1,\mathcal {T}}\le  C\inf_{w\in
\mathbb{U}_\mathcal{T}}\left(\|u-w\|_0 + \|u-w\|_{1,\mathcal {T}}+h|u|_{2,\mathcal{T}}\right).
\end{equation}
The desired error estimate inequality (\ref{CRerror-estimate-inequality}) of this theorem is derived from (\ref{CRconvergence1}) and the interpolation approximation error of the FE space.
\end{proof}

\section{The Hybrid Wilson FVM}

The hybrid Wilson FVM employs the classical Wilson finite element space as its trial space and test space is panned by the characteristic functions of the control volumes in the dual partition combined with certain linearly independent
functions of the trial spaces.

For simplicity, we assume that $\Omega=[a,b]\times [c,d]$.
In the hybrid Wilson FVM, the partition $\mathcal {T}$ is a rectangle partition of $\Omega$:
$
a=x_0 <x_1<\ldots<x_{m_1}=b,\ c=y_0<y_1<\ldots<y_{m_2}.
$
%In each rectangular element in $\mathcal {T}$, we connect the midpoints of every two opposite sides and re-divide $\Omega$ into a sum of some
%other rectangles. Each vertex $P$ of a rectangle element in $\mathcal {T}$ has a surrounding small rectangle (or possibly a small polygon if $P\in \partial\Omega$), called a dual element and denoted by K
%P0 (cf. Fig. 1).
%For $m > 2$ and $m$ points $Q_j\in \mathbb{R}^2$, $j\in \mathbb{N}_m$,
%We use
%$
%\Theta\{Q_1, Q_2, \ldots, Q_m\}
%$
%for the rectangle with the vertices $Q_i$, $i\in \mathbb{N}_m$ being connected consecutively.
For a positive integer $m$, we let $\mathbb{N}_{m}:=\{1,2,\ldots,m\}$.
We use
$
\Theta\{P_1, P_2, P_3, P_4\}
$
for the rectangle with the vertices $P_i$, $i\in \mathbb{N}_4$ being connected consecutively.
For a vertex $P$ of a rectangle element in $\mathcal {T}$, suppose that it is the common vertex of the rectangle elements $K_i\in \mathcal {T}, i\in \mathbb{N}_4$ and suppose that $Q_i,i\in \mathbb{N}_4$ are the centers of $K_i$. The the rectangle $\Theta\{Q_1, Q_2, Q_3, Q_4\}$ is the control volume surrounding the vertex $P$, denoted by  $K^*_P$ (cf. Figure \ref{fig:Wilson_dual_partition1}). For $P\in \partial\Omega$, we derive a control volume associated with it similarly. Then each vertex are associated with a control volume and all control volumes form the dual partition $\mathcal {T}^*$.

\begin{figure}[ht!]
\centering
\includegraphics[width=0.35\textwidth]{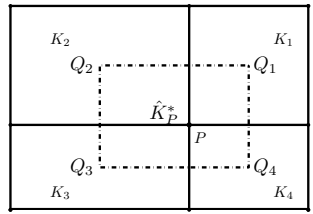}
\caption{A control volume of the hybrid Wilson FVM} \label{fig:Wilson_dual_partition1}
\end{figure}

We choose the square $\hat{K}$ with vertices $\hat{P}_1:=(1,1)$, $\hat{P}_2:=(-1,1)$, $\hat{P}_3:=(-1,-1)$ and $\hat{P}_4:=(1,-1)$ as the reference rectangle. For each $K\in \mathcal {T}$, there is an invertible affine mapping $\mathcal {F}_K$ from $\hat{K}$ to $K$ (cf. \cite{Sh}).
Similar to the FE method, we only need to describe the trial space and the test space on the reference rectangle for the FVMs.
The trial space $\mathbb{U}_{\hat{K}}$ on $\hat{K}$ is a space of polynomials of degree less than or equal to 2. The set of degrees of freedom $\hat{\Sigma}:=\{\hat{\eta}_i:i\in \mathbb{N}_6\}$, where
\begin{equation}\label{freedomfucntional_wil}
\hat{\eta}_i (w) = w(\hat{P}_i), \ i\in \mathbb{N}_4, \quad \hat{\eta}_{4+j}(w)=\int_{\hat{K}} \partial_{jj}w, \ j\in \mathbb{N}_2.
\end{equation}
There is a basis $\hat{\Phi}:=\{\hat{\phi}_i: i\in \mathbb{N}_6\}$ for $\mathbb{U}_{\hat{K}}$ such that
$$
\hat{\eta}_i (\hat{\phi}_j)=
\delta_{i,j}:= \left\{
\begin{array}{ll} 1, & i=j,\\0, &i\neq j,
\end{array}   \quad i,  j\in \mathbb{N}_{6}.
\right.
$$
By simple calculation, we get that
$$
\begin{array}{lllll}
\hat{\phi}_1:= (1/4) (1+x_1)(1+x_2), & \hat{\phi}_2:= (1/4) (1-x_1)(1+x_2), & \hat{\phi}_3:= (1/4) (1-x_1)(1-x_2),\\
\hat{\phi}_4:= (1/4) (1+x_1)(1-x_2),&  \hat{\phi}_5:= (1/8) (x_1^2-1), & \hat{\phi}_6:= (1/8) (x_2^2-1).
\end{array}
$$

\begin{figure}[ht!]
\centering
\includegraphics[width=0.3\textwidth]{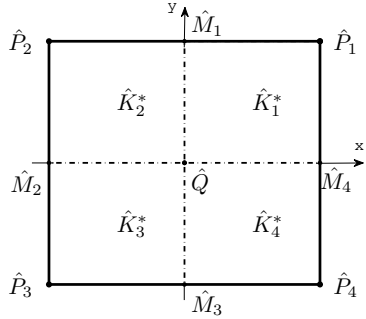}
\caption{The dual partition of the hybrid Wilson FVM on the reference rectangle} \label{fig:Wilson_dual_partition}
\end{figure}

Let $\hat{M}_1:=(0,1)$, $\hat{M}_2:=(-1,0)$, $\hat{M}_3:=(0,-1)$, $\hat{M}_4:=(1,0)$ and $\hat{Q}:=(0,0)$.
The dual partition $\hat{\cal T}^*:=\{\hat{K}^*_i:
i \in \mathbb{N}_4\}$ of $\hat{K}$ is
$$
\begin{array}{lllll}
\hat{K}^*_{1} := \Theta\left\{\hat{Q},\hat{M}_4, \hat{P}_1, \hat{M}_1\right\}, & &  \hat{K}^*_{2} := \Theta\left\{\hat{Q},\hat{M}_1, \hat{P}_2, \hat{M}_2\right\},\\
\hat{K}^*_{3} := \Theta\left\{\hat{Q},\hat{M}_2, \hat{P}_3, \hat{M}_3\right\},& &  \hat{K}^*_{4} := \Theta\left\{\hat{Q},\hat{M}_3, \hat{P}_4, \hat{M}_4\right\}.
\end{array}
$$
In Figure \ref{fig:Wilson_dual_partition}, we draw the reference rectangle $\hat{K}$ and the dual partition $\hat{\cal T}^*$ on it.
The test space on $\hat{K}$ is chosen as $\mathbb{V}_{\refts} := \textrm{span}\
\Psi_{\refts}$, where its basis $\Psi_{\refts}$ consists of
$$
\refpsi_i := \chi_{\hat{K}^*_{i}}, i \in \mathbb{N}_4, \quad  \refpsi_{4+i}  := \hat{\phi}_{4+i}, i\in \mathbb{N}_2.
$$
%where $\chi_{E}$ denotes the characteristic function of $E\subset\mathbb{R}^2$.
%Note that $\mathbb{V}_{\refts} \nsubseteq \mathbb{H}^1_{\mathcal{T}^{*}}(\Omega)$.

%For $w\in \mathbb{H}^2_{\mathcal{T}}(\Omega)$ and $v\in \mathbb{V}_{\mathcal {T}^*}$, we define
%$$
%a_{\mathcal {T}}(w,v):= \sum_{K\in \mathcal {T}}a_K(w,v),
%$$
%where
%$$
%a_K(w,v):= \sum\limits_{K^*\in\mathcal{T}^*} \left( \int_{K^*\cap K} \nabla w^T \mathbf{a}
%\nabla v -\int_{\partial
%K^* \cap \mathrm{int} K}
% v(\mathbf{a}\nabla w)\cdot \mathbf{n} \right).
%$$

By making use of the affine mappings between the reference rectangle $\hat{K}$ and rectangles $K\in \mathcal {T}$, we derive a basis $\Phi_\mathcal {T}:=\{\phi_i: i\in \mathbb{N}_n\}$ for $\mathbb{U}_\mathcal {T}$ and a basis $\Psi_{\mathcal {T}^*}:=\{\psi_i: i\in \mathbb{N}_n\}$ for $\mathbb{V}_{\mathcal {T}^*}$.
Note that $\Psi_{\mathcal {T}^*}$ consists of the nonconforming elements of $\Phi_\mathcal {T}$, which are not continuous on the common edge of the adjacent rectangles. Thus, $\mathbb{V}_{\mathcal {T}^*} \nsubseteq \mathbb{H}^1_{\mathcal{T}^{*}}(\Omega)$. Using $\Phi_\mathcal {T}$
and $\Psi_{\mathcal {T}^*}$, we
define a natural invertible linear mapping
$\Pi_{\mathcal{T}^{*}}:\mathbb{U}_{\mathcal{T}} \to
\mathbb{V}_{\mathcal{T}^{*}}$ for any $w=\sum_{i\in \mathbb{N}_n}
w_i\phi_i\in \mathbb{U}_{\mathcal{T}}$ by
\begin{equation}\label{affine mapping}
    \Pi_{\mathcal{T}^{*}}w:=\sum_{i\in \mathbb{N}_n} w_i\psi_i.
\end{equation}

We turn to the convergence analysis of the hybrid Wilson FVM based on Theorem \ref{thm: convergence_theorem}.
%The trial space may be written as the sum of two spaces
%\begin{equation}\label{space_decom}
%\mathbb{U}_{\mathcal{T}}= \mathbb{U}_{1,\mathcal{T}} + \mathbb{U}_{2,\mathcal{T}},
%\end{equation}
%where $\mathbb{U}_{1,\mathcal{T}}$ is the subspace of
%the functions of $\mathbb{U}_{\mathcal{T}}$ equal to zero at the vertices of all $K\in \mathcal {T}$ and
%$\mathbb{U}_{2,\mathcal{T}}$ is the subspace of
%the functions of $\mathbb{U}_{\mathcal{T}}$ for which the degrees of
%freedom associated with mean values of second derivatives on each element are
%set equal to zero.
The trial space $\mathbb{U}_{\hat{K}}$ on $\hat{K}$ may be written as the sum of two spaces
$\mathbb{U}_{\hat{K}}=\mathbb{U}_{1,\hat{K}}+\mathbb{U}_{2,\hat{K}}$, where $\mathbb{U}_{1,\hat{K}}:=\text{span}\{\hat{\phi}_i: i\in \mathbb{N}_4\}$ and $\mathbb{U}_{2,\hat{K}}:=\text{span}\{\hat{\phi}_5,\hat{\phi}_6\}$. By virtue of this decomposition, every function $w\in \mathbb{U}_{\mathcal{T}}$ consists of two parts
\begin{equation}\label{trial_spa_dec}
w= w_1 +w_2,
\end{equation}
where for each $K\in \mathcal {T}$, $w_1|_K\cdot \mathcal {F}_K \in \mathbb{U}_{1,\hat{K}}$ and $w_2|_K\cdot \mathcal {F}_K \in \mathbb{U}_{2,\hat{K}}$.
Obviously, $w_1$ is uniquely determined by the values of $w$ at the vertices of all $K\in \mathcal {T}$, so that $w_1$ is a continuous function on $\bar{\Omega}$, representing the conforming part of $w$. The function $w_2$ depending merely on the mean values of the second derivatives on each $K\in \mathcal {T}$, is discontinuous at the interelement boundaries and thus nonconforming.
According to (\ref{trial_spa_dec}), from the definition of $\Pi_{\mathcal{T}^{*}}$, we have that
\begin{equation}\label{test_spa_dec}
\Pi_{\mathcal{T}^{*}} w= \Pi_{\mathcal{T}^{*}}w_1 +w_2.
\end{equation}
The function $\Pi_{\mathcal{T}^{*}}w_1$ is a piecewise constant function with
respect to $\mathcal {T}^*$ and its values at vertices of $K\in \mathcal {T}$ are equal to those of $w_1$.

The following lemma is derived from (3.13) of \cite{Sh}.
\begin{lemma}\label{lemma_norm_equi1}
If $\mathscr{T}$ is regular, then there exist positive constants $C_1$ and $C_2$ such that for all $\mathcal {T}\in \mathscr{T}$, all $K\in \mathcal {T}$
and all $w\in \mathbb{U}_{\mathcal{T}}$
\begin{equation}\label{lemma_norm_equi1_mid2}
|w_1|_{1,K} \leq C_1 |w|_{1,K}, \quad  |w_2|_{1,K} \leq C_2 |w|_{1,K}
\end{equation}
where $w_1$ and $w_2$ are the two parts of $w$ as defined in (\ref{trial_spa_dec}).
\end{lemma}

The next lemma is proved in Lemma 2.3 of \cite{LeZ}.
\begin{lemma}\label{lemma_norm_equi2}
If $\mathscr{T}$ is regular, then there exists a positive constant $C$ such that for all $\mathcal {T}\in \mathscr{T}$, all $K\in \mathcal {T}$
and all $w\in \mathbb{U}_{\mathcal{T}}$
\begin{equation}\label{lemma_norm_equi2_mid1}
\|{w}_2\|_{0,K} \leq C h_K |w|_{1,K},
\end{equation}
where $w_2$ is the nonconforming part of $w$ as defined in (\ref{trial_spa_dec}).
\end{lemma}

For each function $w$ defined on $K$, we associate a function $\hat{w}$ defined on $\hat{K}$ by
\begin{equation}\label{funcion_K}
\hat{w}:= w \cdot \mathcal {F}_K.
\end{equation}
The following lemma is derived from (2.5) of \cite{Sh}.
\begin{lemma}\label{lemma_norm_equi3}
If $\mathscr{T}$ is regular, then there exist positive constants $C_1$ and $C_2$ such that for all $\mathcal {T}\in \mathscr{T}$, all $K\in \mathcal {T}$
and all $w\in \mathbb{H}^1(K)$
$$
C_1 |w|_{1,K} \leq|\hat{w}|_{1,\hat{K}} \leq C_2|w|_{1,K},
$$
\end{lemma}

For each $K\in \mathcal {T}$, we denote its vertices by $P_{i,K}, i\in \mathbb{N}_4$ anticlockwise and set $P_{5,K}:=P_{1,K}$.
In the following proposition, we establish inequality (\ref{boundedness_dis_norm1}) for the hybrid Wilson FVM.
\begin{proposition}\label{norm_eqi_wilson}
If $\mathscr{T}$ is regular, then for all $w\in \mathbb{U}_{\mathcal{T}}$ there holds
\begin{equation}\label{eqnorm_wilson}
|\Pi_{\mathcal{T}^{*}}w|_{1,\mathbb{V}_{\mathcal {T}^*}} \le C
\|w\|_{1,\mathcal {T}}.
\end{equation}
\end{proposition}
\begin{proof}
For each $\mathcal{T}\in\mathscr{T}$ and each $w\in \mathbb{U}_{\mathcal{T}}$, let $w^*:=\Pi_{\mathcal{T}^{*}}w$.
By (\ref{trial_spa_dec}) and (\ref{test_spa_dec}), we have that $w=w_1+w_2$ and $w^*=w_1^* + w_2$, where $w_1^*:=\Pi_{\mathcal{T}^{*}}w_1$. To derive the desired inequality (\ref{eqnorm_wilson}) of this lemma, it suffices to prove that
\begin{equation}\label{norm_eqi_wilson_mid1}
|w_1^*|_{1,\mathbb{V}_{\mathcal {T}^*},K}^2 \leq C |w|_{1,K}^2,  \quad |w_2|_{1,\mathbb{V}_{\mathcal {T}^*},K}^2 \leq C |w|_{1,K}^2.
\end{equation}

We begin to prove the first inequality of (\ref{norm_eqi_wilson_mid1}).  Note that
$$
|w_1^*|_{1,\mathbb{V}_{\mathcal {T}^*},K}^2 = \sum_{\ell^*\in
L^*_K}|\ell^*|^{-1}\int_{\ell^*} [w_1^*]^2 = \sum_{i\in \mathbb{N}_4} \left(w_1(P_{i,K})-w_1(P_{i+1,K})\right)^2,
%(w(P_1)-w(P_2))^2 + (w(P_2)-w(P_3))^2  + (w(P_3)-w(P_4))^2  +(w(P_4)-w(P_1))^2.
$$
Since $\sum_{i\in \mathbb{N}_4} \left(w_1(P_{i,K})-w_1(P_{i+1,K})\right)^2$ and $|\hat{w}_1|_{1,\hat{K}}^2$ are nonnegative quadratic forms of $w_1(P_{i,K}), i\in \mathbb{N}_4$ and they have the same null space, it follows from \cite{HJ} that they are equivalent. Thus, there exists a positive constant $C$ independent of grids such that
\begin{equation}\label{norm_eqi_wilson_mid2}
|w_1^*|_{1,\mathbb{V}_{\mathcal {T}^*},K}^2 \leq C |\hat{w}_1|_{1,\hat{K}}^2.
\end{equation}
Combining (\ref{norm_eqi_wilson_mid2}) and Lemma \ref{lemma_norm_equi3} gives that
\begin{equation}\label{norm_eqi_wilson_mid3}
|w_1^*|_{1,\mathbb{V}_{\mathcal {T}^*},K}^2 \leq C |{w}_1|_{1,K}^2.
\end{equation}
Then, the first inequality of (\ref{norm_eqi_wilson_mid1}) is derived from (\ref{norm_eqi_wilson_mid3}) and the first inequality of (\ref{lemma_norm_equi1_mid2}) in Lemma \ref{lemma_norm_equi1}.

%We next prove the second inequality of (\ref{norm_eqi_wilson_mid1}).
Since $w_2$ is continuous on each $K\in \mathcal {T}$, we observe that
$$
|w_2|_{1,\mathbb{V}_{\mathcal {T}^*},K}^2 = \sum_{K^*\in\mathcal{T}^*}
|w_2|_{1,K^*\cap K}^2 =|w_2|_{1,K}^2.
$$
The above equation and the second inequality of (\ref{lemma_norm_equi1_mid2}) yield the second inequality of (\ref{norm_eqi_wilson_mid1}).
\end{proof}

We verify inequality (\ref{boundedness_dis_norm2}) for the hybrid Wilson FVM in the next proposition.
\begin{proposition}\label{L2testspace_wilson}
If $\mathscr{T}$ is regular and $h<1$, then for all $\mathcal{T}\in\mathscr{T}$ and all $w\in \mathbb{U}_{\mathcal{T}}$ there holds
$$
\|\Pi_{\mathcal{T}^{*}}w\|_{0,\Omega} \leq C \|w\|_{1,\mathcal {T}}.
$$
\end{proposition}
\begin{proof}
For each $\mathcal{T}\in\mathscr{T}$ and each $w\in \mathbb{U}_{\mathcal{T}}$, let $w^*:=\Pi_{\mathcal{T}^{*}}w$.
According to(\ref{trial_spa_dec}) and (\ref{test_spa_dec}), we have the decomposition $w=w_1+w_2$ and $w^*=w_1^* + w_2$, where $w_1^*:=\Pi_{\mathcal{T}^{*}}w_1$.

For each $K\in \mathcal {T}$ and each $K^*\in \mathcal {T}^*$, since $w_1^*$ is constant on $K\cap K^*$, we observe that
$$
\|w_1^*\|_{0,K}^2= \frac{1}{4} |K|\cdot \sum_{i\in \mathbb{N}_4} w_1^2(P_{i,K}).
$$
By changing variables, we get that
$$
\|w_1\|_{0,K}^2 = \frac{1}{4}|K|\cdot \|\hat{w}_1\|_{0,\hat{K}}^2.
$$
By simple calculation, we know that both $\sum_{i\in \mathbb{N}_4} w_1^2(P_{i,K})$ and $|\hat{w}_1|_{0,\hat{K}}^2$ are positive definite quadratic forms of $w_1(P_{i,K}), i\in \mathbb{N}_4$. Thus, there exists a positive constant $C_1$ independent of meshes such that
\begin{equation}\label{L2testspace_wilson_mid1}
\|w_1^*\|_{0,K}^2 \leq C_1^2 \|w_1\|_{0,K}^2.
\end{equation}
From (\ref{L2testspace_wilson_mid1}), the Poincar$\acute{e}$ inequality and the first inequality of (\ref{lemma_norm_equi1_mid2}) in Lemma \ref{lemma_norm_equi1}, we derive that
\begin{equation}\label{L2testspace_wilson_mid2}
\|w_1^*\|_{0,\Omega} \leq C_1 \|w_1\|_{0,\Omega} \leq C_1 |w_1|_{1,\Omega} \leq C \|w\|_{1,\mathcal {T}}.
\end{equation}
It follows from Lemma \ref{lemma_norm_equi2} that
\begin{equation}\label{L2testspace_wilson_mid3}
\|w_2\|_{0,\Omega} \leq C h \|w\|_{1,\mathcal {T}}.
\end{equation}
Since $h<1$, from (\ref{L2testspace_wilson_mid2}) and (\ref{L2testspace_wilson_mid3}), we conclude that
$$
\|w^*\|_{0,\Omega} \leq \|w_1^*\|_{0,\Omega} + \|w_2\|_{0,\Omega} \leq C\|w\|_{1,\mathcal {T}},
$$
which proves the desired inequality of this proposition.
\end{proof}

We estimate the nonconforming error as defined in (\ref{nonconforming error}) for the hybrid Wilson FVM.
\begin{proposition}\label{E_wilson}
Let $u\in \mathbb{H}_0^1(\Omega)\cap \mathbb{H}^2(\Omega)$ be the solution of
(\ref{eq:poisson_equation}) and $u_\mathcal{T}\in
\mathbb{U}_\mathcal{T}$ be the solution of the hybrid Wilson FVM equation.
If $\mathscr{T}$ is regular, then there exists a positive constant $c$ such that for all $\mathcal{T}$ and all $v\in \mathbb{U}_{\mathcal{T}}$
$$
|E_\mathcal {T}(u,v)| \leq c h |u|_{2} \|v\|_{1, \mathcal {T}}.
$$
\end{proposition}
\begin{proof}
For each $v\in \mathbb{U}_{\mathcal{T}}$, we let $v^*:= \Pi_{\mathcal {T}^*}v$. From (\ref{test_spa_dec}), we have that $v^*=v^*_1+v_2$, where $v^*_1$ is a piecewise constant function with
respect to $\mathcal {T}^*$ and $v_2\in \mathbb{U}_{\mathcal{T}}$.
From the definition, we get that
\begin{equation}\label{E_wilson_mid1}
    E_\mathcal {T}(u,v)=a_{\mathcal {T}}(u, v^*)- (f, v^*) = a_{\mathcal {T}}(u, v^*_1)+a_{\mathcal {T}}(u, v^*_2) - (f, v^*_1)-(f, v^*_2).
\end{equation}
Since $v^*_1\in \mathbb{H}^1_{\mathcal{T}^{*}}(\Omega)$, from (\ref{variationaleq}), we note that
\begin{equation}\label{E_wilson_mid2}
    a_{\mathcal {T}}(u, v^*_1) = (f, v^*_1).
\end{equation}
Combining (\ref{E_wilson_mid1}) and (\ref{E_wilson_mid2}) yields
$$
E_\mathcal {T}(u,v) = a_{\mathcal {T}}(u, v_2)-(f, v_2).
% = \sum_{K\in\mathcal {T}} \int_{K} (\nabla u)^T \mathbf{a} \nabla v^*_2 -(f,v^*_2).
$$
Since $v_2\in \mathbb{U}_{\mathcal{T}}$ and $\mathscr{T}$ is regular, employing the result in the nonconforming FE method yields the desired inequality of this proposition (cf. \cite{Ciarlet}).
\end{proof}

We introduce an interpolation projection operator to the trial space. For any function $\hat{v}\in \mathbb{H}^2(\hat{K})$, we define the interpolation function $\hat{P}\hat{v} \in \mathbb{U}_{\hat{K}}$ as follows
$$
\hat{\eta}_i (\hat{P}\hat{v}) = \hat{\eta}_i (\hat{v}), \quad i\in \mathbb{N}_6,
$$
where $\hat{\eta}_i$ are defined as in (\ref{freedomfucntional_wil}). Then, for any function $v\in \mathbb{H}^2({K})$, the corresponding function $P_K v$ is defined by
$$
\widehat{P_K v} = \hat{P} \hat{v}, \quad  v=\hat{v} \cdot \mathcal {F}_K^{-1}.
$$
For each $v\in \mathbb{H}^2(\Omega)$, let the interpolation function $P_\mathcal {T} v \in \mathbb{U}_{\mathcal{T}}$ be such that
$$
 P_\mathcal {T} v|_K =P_K v,  \quad  \mbox{for any} \ K\in \mathcal {T}.
$$
By virtue of the decomposition (\ref{trial_spa_dec}), the interpolation function $P_\mathcal {T} v$ can be written as the sum of the
conforming part denoted by $Q_\mathcal {T} v$ and the nonconforming part denoted by $R_\mathcal {T}v$, that is,
\begin{equation}\label{projection}
 P_\mathcal {T} v = Q_\mathcal {T} v + R_\mathcal {T}v.
\end{equation}

The interpolation error estimates presented in the next lemma are derived from (5.16) and (5.17) of \cite{Sh}.
\begin{lemma}\label{interpolation_error}
For any $v\in \mathbb{H}^2(\Omega)$, there holds
$$
\|v- P_\mathcal{T} v  \|_{1,\mathcal{T}} \leq Ch|v|_2, \quad \|v- P_\mathcal{T} v  \|_{0} \leq Ch^2 |v|_2,
$$
and
$$
|v- Q_\mathcal{T} v  |_{1} \leq Ch|v|_2, \  \quad \|v- Q_\mathcal{T} v  \|_{0} \leq Ch^2 |v|_2.
$$
\end{lemma}

We are ready to get the convergence theorem for the hybrid Wilson FVM.

\begin{theorem}\label{thm:WilsonConvergence}
Let $u\in \mathbb{H}_0^1(\Omega)\cap \mathbb{H}^2(\Omega)$ be the solution of
(\ref{eq:poisson_equation}).
Suppose that $\mathscr{T}$ is regular. If the family $\mathscr{A}_\mathscr{T}$ of the discrete bilinear
forms is
uniformly elliptic, then for each $\mathcal{T}\in \mathscr{T}$ the hybrid Wilson FVM equation has a unique solution $u_\mathcal{T}\in
\mathbb{U}_\mathcal{T}$, and there exists a positive constant $C$
such that for all $\mathcal{T}\in \mathscr{T}$
$$
\|u-u_\mathcal{T}\|_{1,\mathcal {T}}\le C h |u|_{2}.
$$
\end{theorem}
\begin{proof}
From Theorem \ref{thm: convergence_theorem}, Propositions \ref{norm_eqi_wilson}, \ref{L2testspace_wilson} and \ref{E_wilson}, we derive that
$$
\|u-u_\mathcal{T}\|_{1,\mathcal {T}}\le C \left(\inf_{w\in \mathbb{U}_\mathcal{T}}\left(\|u-w\|_{0,\Omega}+\|u-w\|_{1,\mathcal {T}} + h|u-w|_{2,\mathcal {T}}\right) +  h |u|_{2}\right).
$$
This combined with the interpolation error estimate presented in Lemma \ref{interpolation_error} yields the desired result of this theorem.
\end{proof}

It can be seen from Theorem \ref{thm:WilsonConvergence} that the hybrid Wilson FVM enjoys the same order of error estimate as that of the Wilson FEM (\cite{LeZ, Sh}).
We have seen in Theorem \ref{thm:WilsonConvergence} that the uniform
ellipticity of $\mathscr{A}_\mathscr{T}$ is crucial to
obtain the error estimate of the hybrid Wilson FVM.
The rest of this section is devoted to establishing the uniform ellipticity of $\mathscr{A}_\mathscr{T}$ for the case that the matrix $\mathbf{a}$  in (\ref{eq:poisson_equation}) is
chosen as the identity matrix and $b=0$.
%We will prove that if the family $\mathscr{T}$ of the  rectangle partitions is regular, then $\mathscr{A}_\mathscr{T}$ is uniformly elliptic.
In order to prove the uniform ellipticity inequality (\ref{uniform_elliptic}), it suffices to verify that there exists a constant $\sigma>0$ such that for all $\mathcal {T}\in \mathscr{T}$ and the associated $\mathcal {T}^*$ and all $K \in
\mathcal{T}$,
\begin{equation}\label{local_elliptic}
a_{K}(w, \Pi_{\mathcal{T}^{*}}w)\ge \sigma
|w|^2_{1,K}, \quad  \mbox{for all} \
w\in \mathbb{U}_\mathcal {T}
\end{equation}

For each $K\in \mathcal {T}$, we define a discrete semi-norm for $\mathbb{U}_{\mathcal{T}}$ restricted on $K$.
According to the FE theory (cf. \cite{Ciarlet}),
for each $K\in \mathcal {T}$, corresponding to the FE triple element $(\hat{K}, \hat{\Sigma}, \mathbb{U}_{\hat{K}})$ on the reference triangle $\hat{K}$, there is a FE triple element $(K, {\Sigma}_K, \mathbb{U}_{K})$ on $K$.
Note that the set of degrees of freedom
$
{\Sigma}_K:=\{{\eta}_{i,K}: i\in \mathbb{N}_{6}\}
$
are the functionals corresponding to $\hat{\eta}_i$ in the sense that for all ${w}\in \mathbb{U}_{\mathcal{T}}$,
$
    {\eta}_{i,K}(w) = \hat{\eta}_i(\hat{w}).
$
For each $w\in
\mathbb{U}_{\mathcal{T}}$ and $K\in \mathcal {T}$, we
let
\begin{equation}\label{coef_on_K}
    w_{i,K}:= {\eta}_{i,K}(w),\ i\in \mathbb{N}_{6}, \quad
\bar{{w}}_K:= \frac{1}{4}\sum_{i\in \mathbb{N}_{4}}
{w}_{i,K}\quad  \mbox{and} \quad \bar{w}_{i,K}:={w}_{i,K}-\bar{{w}}_K,\ i\in \mathbb{N}_{4}.
\end{equation}
Define
$$
|w|_{1,\mathbb{U}_{\mathcal{T}},K}: = \isqrt{
\sum_{i=1}^4\bar{w}_{i,K}^2 +
{w}_{5,K}^2 + {w}_{6,K}^2}.
$$
Similar to the proof of Lemma 3.3 of \cite{CXZ}, we derive that if $\mathscr{T}$ is regular, there exist positive constants $c_1$ and $c_2$ such that for all $\mathcal{T}\in \mathscr{T}$ and all $K\in \mathcal{T}$,
\begin{equation}\label{eqnorm2_mid2}
  c_1  |w|_{1,\mathbb{U}_{\mathcal{T}},K}  \leq  |w|_{1,K}  \leq c_2 |w|_{1,\mathbb{U}_{\mathcal{T}},K},
\end{equation}

We reexpress (\ref{local_elliptic}) in an equivalent matrix form.
To this end, for each $K\in \mathcal {T}$,
we define
\begin{equation}\label{Wilson-trial basis on K}
\phi_{i,K}:=\hat{\phi}_i \circ \mathcal{F}_K^{-1}, \   \psi_{i,K}:=\hat{\psi}_i \circ \mathcal{F}_K^{-1},\quad  i\in \mathbb{N}_{6}
\end{equation}
and
$$
\mathbf{A}_K := \left[a_{K}\left(\phi_{i,K},
\psi_{j,K}\right):i,j\in \mathbb{N}_{6}\right], \quad
\tilde{\mathbf{A}}_K:=
(\mathbf{A}_K + \mathbf{A}_K^T)/2.
$$
The matrix $\tilde{\mathbf{A}}_K$ is the symmetrization of the element stiffness matrix $\mathbf{A}_K$.
Note that for each $w\in \mathbb{U}_\mathcal{T}$ and each $K\in \mathcal {T}$,
\begin{equation}\label{w_express}
    w(x)=\sum_{i\in
\mathbb{N}_{6}} w_{i,K} \phi_{i,K}(x)\quad \text{and} \quad  \Pi_{\mathcal{T}^{*}}w(x)=\sum_{i\in
\mathbb{N}_{6}} w_{i,K} \psi_{i,K}(x),  \quad \forall x\in K,
\end{equation}
where $w_{i,K}, i\in \mathbb{N}_{6}$ are as defined in (\ref{coef_on_K}).
For each $w\in \mathbb{U}_\mathcal{T}$ and each $K\in \mathcal {T}$, we let
$$
    \mathbf{w}_K := \left[{w}_{i,K}: i\in
\mathbb{N}_{6}\right]^T.
$$
We define a matrix of rank 1 by setting
$$
{\mathbf{e}}:=\left[1,1,1,1,0,0\right]^T \quad \text{and} \quad   \mathbf{E} := \frac{1}{{\mathbf{e}}^T
{\mathbf{e}}} {\mathbf{e}} {\mathbf{e}}^T.
$$
Note that the rank of $\mathbf{E}$ is one and $\mathbf{e}$ is an eigenvector of $\mathbf{E}$ associated with the eigenvalue 1.
Furthermore, note that for each $w\in \mathbb{U}_\mathcal{T}$
\begin{equation}\label{lm:ule_matrix_form_mid2}
    |w|^2_{1,\mathbb{U}_{\mathcal{T}},K} = (\mathbf{w}_K-\mathbf{E}\mathbf{w}_K)^T(\mathbf{w}_K-\mathbf{E}\mathbf{w}_K).
\end{equation}
From (\ref{eqnorm2_mid2}) and (\ref{lm:ule_matrix_form_mid2}), we obtain the following result as a lemma.

\begin{lemma}\label{lm:ule_matrix_formWilson}
If $\mathscr{T}$ is regular, then (\ref{local_elliptic}) is equivalent to the existence of a positive constant $\sigma$ such that for all $\mathcal{T}\in \mathscr{T}$, all $K\in\mathcal{T}$ and all $w\in \mathbb{U}_\mathcal{T}$,
\begin{equation}\label{K-elliptic1}
\mathbf{w}_K^T\tilde{\mathbf{A}}_K \mathbf{w}_K\ge \sigma
(\mathbf{w}_K-\mathbf{E}\mathbf{w}_K)^T(\mathbf{w}_K-\mathbf{E}\mathbf{w}_K).
\end{equation}
\end{lemma}

We next express the element stiffness matrices $\mathbf{A}_K$ for all $K\in\mathcal {T}$ in terms of two matrices on the reference rectangle $\hat K$.
For ${\phi}\in \mathbb{U}_{\hat{K}}$, ${\psi}\in \mathbb{V}_{\hat{\mathcal{T}}^*}$ and $x:=(x_1,x_2)\in \hat{K}$,
let
\[
{a}_{1}({\phi}, {\psi})  :=  \sum_{\hat{K}^*\in \hat{\mathcal{T}}^*} \left(
\int_{\hat{K}^*}\frac{\partial {\phi(x)}}{\partial {x}_1} \frac{\partial {\psi}(x)}{\partial {x}_1}dx_1 dx_2
- \int_{\partial \hat{K}^* \cap \mathrm{int}\hat{K}} {\psi(x)} \frac{\partial{\phi(x)}}{\partial {x}_1}dx_2 \right),
\]
and
$$
{a}_{2}({\phi}, {\psi})  :=  \sum_{\hat{K}^*\in \hat{\mathcal{T}}^*} \left(
\int_{\hat{K}^*}\frac{\partial {\phi(x)}}{\partial {x}_2} \frac{\partial {\psi}(x)}{\partial {x}_2}dx_1 dx_2
- \int_{\partial \hat{K}^* \cap \mathrm{int}\hat{K}} {\psi(x)} \frac{\partial{\phi(x)}}{\partial {x}_2}dx_1 \right).
$$
For the basis $\Phi_{\hat{K}}= \{\hat{\phi}_i: i\in \mathbb{N}_{6}
\}$ of $\mathbb{U}_{\hat{K}}$ and the basis $\Psi_{\hat{\mathcal {T}}^*}=\{\hat{\psi}_j: j\in
\mathbb{N}_{6 }\}$ of $\mathbb{V}_{\hat{\mathcal{T}}^*}$, let
%\begin{equation}\label{matrices-reference}
$$
{\mathbf{A}}_{i} :=
\left[{a}_{i}\left(\hat{\phi}_l,\hat{\psi}_m\right):l,m\in \mathbb{N}_{6}\right].
$$
%\end{equation}
For each $K\in \mathcal {T}$, we use $2h_{1,K}$ and $2h_{2,K}$ to denote the lengths of the edges parallel to the $x$-axis and the $y$-axis respectively and define the shape parameter of $K$
\begin{equation}\label{ratio}
r_K:= h_{2,K}/h_{1,K}.
\end{equation}
Obviously, the regularity condition (\ref{regularity}) of the family $\mathscr{T}$ of the rectangle partitions is equivalent to that there exist positive constants $\lambda_1$ and $\lambda_2$ such that for all $\mathcal {T}\in \mathscr{T}$ and all $K\in \mathcal {T}$
\begin{equation}\label{regular_inequlatywilson}
 \lambda_1  \leq  r_K \leq \lambda_2.
\end{equation}
We introduce a matrix
$$
\mathbf{M}_{K} := \left[
\begin{array}{cc}r_{K} & 0 \\ 0 &
1/r_{K}\end{array}\right].
$$
%Note that the nonzero elements of $\mathbf{M}_{K}$ are the ratios of the lengths of the adjacent edges of $K$.

\begin{lemma}
\label{lm:matrix_parameterizationWilson} For each rectangle element $K$,
\[
\mathbf{A}_K = r_K {\mathbf{A}}_{1} + (1/r_{K}) {\mathbf{A}}_{2} .
\]
\end{lemma}
\begin{proof}
Recall for $\phi_{i,K}$ and $\psi_{j,K}$ defined as in (\ref{Wilson-trial basis on K}) that
\[
 a_K \left(\phi_{i,K}, \psi_{j,K}\right)
=  \sum_{K^*\in\mathcal{T}^*} \left\{ \int_{K^*\cap K} \nabla \phi_{i,K}
\cdot \nabla \psi_{j,K} d{x} - \int_{\partial K^* \cap K} \psi_{j,K} \nabla \phi_{i,K}
\cdot \mathbf{n} \ ds \right\}.
\]
Using the affine mapping between the reference rectangle $\hat{K}$ and $K$, we derive
\begin{equation}\label{ak}
 a_K \left(\phi_{i,K}, \psi_{j,K}\right) =
\sum_{K^*\in\mathcal{T}^*} \left(
 \int_{\hat{K} \cap \hat{K}^*} (\nabla \hat{\phi}_{i})^T
\mathbf{M}_{K} \nabla \hat{\psi}_{j} \ d\hat{x}  - \int_{
\partial \hat{K}^* \cap
\hat{K}} \hat{\psi}_{j} (\nabla \hat{\phi}_{i})^T \mathbf{M}_{K} \hat{\bf n}
\ d\hat{s} \right).
\end{equation}
Substituting the definition of $\mathbf{M}_{K}$ into (\ref{ak}), we
obtain the desired result of this lemma.
\end{proof}

Let
$$
\tilde{\mathbf{A}}_i:= (\mathbf{A}_i + \mathbf{A}_i^T)/2, \ i=1,2.
$$
For each $r\in \mathbb{R}$, we introduce a matrix
$$
\mathbf{H}\left(r \right) := r \tilde{\mathbf{A}}_{1} + (1/r) \tilde{\mathbf{A}}_{2} + {\mathbf{E}}.
$$
From Lemma \ref{lm:matrix_parameterizationWilson}, we learn that for each $K\in\mathcal{T}$
$$
\mathbf{H}\left(r_K \right) :=  \tilde{\mathbf{A}}_{K}  + {\mathbf{E}}.
$$
The next lemma presents a sufficient condition for (\ref{K-elliptic1}) by making use of $\mathbf{H}\left(r_K \right)$

\begin{lemma}\label{Wilson eigenvalue}
If there exists a positive constant $c$ such that for all $\mathcal{T}\in \mathscr{T}$ and all $K\in\mathcal{T}$
\begin{equation}\label{Wilson eigenvalue mid1}
    \lambda_{\min}(\mathbf{H}\left(r_{K}\right))\geq c,
\end{equation}
then (\ref{K-elliptic1}) holds.
\end{lemma}
\begin{proof}
If (\ref{K-elliptic1}) does not hold, then for any $\sigma >0$, there exist a $\mathcal{T}\in \mathscr{T}$, a $K\in\mathcal{T}$ and a $w\in \mathbb{U}_\mathcal{T}$ such that
\begin{equation}\label{Wilson eigenvalue mid2}
    \mathbf{w}_K^T\tilde{\mathbf{A}}_K \mathbf{w}_K< \sigma
(\mathbf{w}_K-\mathbf{E}\mathbf{w}_K)^T(\mathbf{w}_K-\mathbf{E}\mathbf{w}_K).
\end{equation}
Let $V_1:=\mathrm{span} \{\mathbf{e}\}$ and $V_2:=\mathrm{span}\{\mathbf{v}_i, i\in
\mathbb{N}_{5}\}$ where $\mathbf{v}_i, i\in \mathbb{N}_{5}$ are the orthogonal
eigenvectors of $\mathbf{E}$ associated with the eigenvalue 0. Then, $V_1$ is the eigen-space of $\mathbf{E}$ associated with the eigenvalue 1 and $V_2$ is the eigen-space of $\mathbf{E}$ associated with the eigenvalue 0.

We prove that $V_1$ is contained in the null space of $\tilde{\mathbf{A}}_K$. From Lemma \ref{lm:matrix_parameterizationWilson}, we get
\begin{equation}\label{Wilson eigenvalue mid3}
\tilde{\mathbf{A}}_K \mathbf{e} = \frac{r_K}{2}(\mathbf{A}_1 \mathbf{e} + \mathbf{A}_1^T \mathbf{e}) + \frac{1}{2 r_K}(\mathbf{A}_2 \mathbf{e} + \mathbf{A}_2^T \mathbf{e}).
\end{equation}
From the definition of $\mathbf{A}_i, i=1,2$, the $k$th elements of the vectors $\mathbf{A}_i \mathbf{e}$ and $\mathbf{A}_i^T \mathbf{e}$ are as follows
$$
\left(\mathbf{A}_i \mathbf{e}\right)_k = a_i(\hat{\phi}_k, \sum_{j=1}^4 \hat{\psi}_j) = a_i(\hat{\phi}_k, \chi_{\hat{K}})=0, \quad \left(\mathbf{A}_i^T \mathbf{e}\right)_k = a_i(\sum_{j=1}^4 \hat{\phi}_j,\hat{\psi}_k) = a_i(\chi_{\hat{K}},\hat{\psi}_k)=0
$$
Thus
$$
\mathbf{A}_i \mathbf{e} = \mathbf{A}_i^T \mathbf{e} =0, \quad i=1,2.
$$
Substituting the above equations into (\ref{Wilson eigenvalue mid3}) yields that $\tilde{\mathbf{A}}_K \mathbf{e}=0$.

Note that there exist $\mathbf{w}_{1,K}\in V_1$ and $\mathbf{w}_{2,K}\in V_2$ such that $\mathbf{w}_K=\mathbf{w}_{1,K}+\mathbf{w}_{2,K}$. Then, we get that
\begin{equation}\label{Wilson eigenvalue mid4}
    \mathbf{w}_K^T\tilde{\mathbf{A}}_K \mathbf{w}_K = \mathbf{w}_{2,K}^T\tilde{\mathbf{A}}_K \mathbf{w}_{2,K}, \quad (\mathbf{w}_K-\mathbf{E}\mathbf{w}_K)^T(\mathbf{w}_K-\mathbf{E}\mathbf{w}_K) = \mathbf{w}_{2,K}^T \mathbf{w}_{2,K}.
\end{equation}
Substituting (\ref{Wilson eigenvalue mid4}) into (\ref{Wilson eigenvalue mid2}), we get that for any $\sigma >0$, there exist a $\mathcal{T}\in \mathscr{T}$, a $K\in\mathcal{T}$ and a $\mathbf{w}_{2,K}\in V_2$ such that
\begin{equation}\label{lm:relations_between_h_and_ak_midd3}
    \mathbf{w}_{2,K}^T\tilde{\mathbf{A}}_K \mathbf{w}_{2,K} < \sigma \mathbf{w}_{2,K}^T \mathbf{w}_{2,K}.
\end{equation}
Since $\sigma$ can be sufficiently small, from (\ref{lm:relations_between_h_and_ak_midd3}), we derive that there exists a $\mathbf{w}_{2,K_0}\in V_2$ such that $\mathbf{w}_{2,K_0}^T\tilde{\mathbf{A}}_K \mathbf{w}_{2,K_0} \leq 0$. Hence,
$$
\mathbf{w}_{2,K_0}^T\mathbf{H}(r_{K_0})\mathbf{w}_{2,K_0} =  \mathbf{w}_{2,K_0}^T\tilde{\mathbf{A}}_{K_0} \mathbf{w}_{2,K_0} \leq 0.
$$
This contradicts (\ref{Wilson eigenvalue mid1}). Therefore, we conclude that (\ref{K-elliptic1}) holds.
\end{proof}

Now we are read to establish the uniform ellipticity of the family of the discrete bilinear forms for the hybrid Wilson FVM.
\begin{theorem}\label{Wilson eigenvaluepositive}
If $\mathscr{T}$ is regular, then $\mathscr{A}_\mathscr{T}$ is uniformly elliptic.
\end{theorem}
\begin{proof}
By Lemmas \ref{lm:ule_matrix_formWilson} and \ref{Wilson eigenvalue}, we only need to prove that there exists a positive constant $c$ independent of meshes such that (\ref{Wilson eigenvalue mid1}) holds.

By simple calculation, we derive that the matrices $\tilde{\mathbf{A}}_i, i=1,2$ are semi-definite with rank 3. Since $\mathscr{T}$ is regular, by (\ref{regular_inequlatywilson}), we learn that
\begin{equation}\label{Wilson eigenvaluepositiveMid1}
\mathbf{H}(r_{K})\geq \lambda_1 \tilde{\mathbf{A}}_1 + (1/{\lambda_2}) \tilde{\mathbf{A}}_2 + \mathbf{E} \geq \min\{\lambda_1, 1/{\lambda_2}, 1\}
(\tilde{\mathbf{A}}_1 + \tilde{\mathbf{A}}_2 + \mathbf{E}).
\end{equation}
It can be directly computed that the minimum eigenvalue of the matrix $\tilde{\mathbf{A}}_1 + \tilde{\mathbf{A}}_2 + \mathbf{E}$ is $\frac{1}{12}$.
Therefore, (\ref{Wilson eigenvalue mid1}) holds with $c:=\frac{1}{12} \min\{\lambda_1, 1/{\lambda_2}, 1\}$.
\end{proof}

%%%%%%%%%%%%%%%%%%%%%%%%%%%%%%%%%%%%%%%%%%%%%%%%%%%%%%%%%%%%%%%%%%%%%%%
\section{The $L^2$ Error Estimate of the Hybrid Wilson FVM}
%%%%%%%%%%%%%%%%%%%%%%%%%%%%%%%%%%%%%%%%%%%%%%%%%%%%%%%%%%%%%%%%%%%%%%%

The $L^2$ error estimate of the C-R FVM for solving the Poisson equation was developed in \cite{CP1}.
In this section, we shall establish the $L^2$ error estimate of the hybrid Wilson FVM for solving the Poisson equation.
The result will show that it is enjoys the same optimal convergent rate in $L^2$ norm as that of the Wilson FEM.

We first present two useful lemmas. The next lemma is obtained from (3.13) of \cite{Sh}.

\begin{lemma}\label{con_and_non_trial}
If $\mathscr{T}$ is regular, then there exist positive constants $C_1$ and $C_2$ such that for each $w\in \mathbb{U}_{\mathcal {T}}$ with conforming part $w_1$  and nonconforming part $w_2$ as defined in (\ref{trial_spa_dec})
$$
 |w_1|_{1} \leq C_1 \|w\|_{1,\mathcal {T}}, \quad    |w_2|_{1,\mathcal {T}} \leq C_2 \|w\|_{1,\mathcal {T}}.
$$
\end{lemma}

According to (\ref{trial_spa_dec}),
the solution $u_{\mathcal {T}}$ of the hybrid Wilson FVM can be written as the sum
\begin{equation}\label{decompositionFVMsol}
u_{\mathcal {T}}= u_{\mathcal {T},c} + u_{\mathcal {T},n},
\end{equation}
where $u_{\mathcal {T},c}$ is the conforming part and $u_{\mathcal {T},n}$ is the nonconforming part of $u_{\mathcal {T}}$.
For each $K\in \mathcal{T}$,
let
$$\mathbf{F}_K:=\frac{1}{|K|}\left[
\begin{array}{cc}h_{2,K}^2 & 0 \\ 0 &
h_{1,K}^2\end{array}\right].
$$

\begin{lemma}\label{nonconf_part}
If $\mathscr{T}$ is regular,
 then there holds
$$
\|u_{\mathcal {T},n}\|_{1,\mathcal {T}} \leq C h |u|_2 , \quad \|u_{\mathcal {T},n}\|_{0,\Omega} \leq C h^2 |u|_2.
$$
\end{lemma}
\begin{proof}
Note that
\begin{equation}\label{nonconf_part_mid1}
\|u_{\mathcal {T},n}\|_{1,\mathcal {T}} \leq \|u_{\mathcal {T}}-u\|_{1,\mathcal {T}} +|u- Q_{\mathcal{T}}u|_1 + |Q_{\mathcal{T}}u- u_{\mathcal {T},c}|_{1}
\end{equation}
where the projection $Q_{\mathcal{T}}$ is defined as in (\ref{projection}).
By Theorems \ref{thm:WilsonConvergence} and \ref{Wilson eigenvaluepositive} and Lemma \ref{interpolation_error}, we get that
\begin{equation}\label{nonconf_part_mid6}
\|u_{\mathcal {T}}-u\|_{1,\mathcal {T}} +|u- Q_{\mathcal{T}}u| \leq Ch|u|_2
\end{equation}
Note that $Q_{\mathcal{T}}u-u_{\mathcal{T},c}$ is the conforming part of $P_{\mathcal{T}}u- u_{\mathcal{T}}$. From Lemma \ref{con_and_non_trial}, we obtain that
\begin{equation}\label{nonconf_part_mid2}
|Q_{\mathcal{T}}u-u_{\mathcal{T},c}|_1 \leq C\|P_{\mathcal{T}}u- u_{\mathcal{T}}\|_{1, \mathcal{T}} \leq C\|P_{\mathcal{T}}u- u\|_{1, \mathcal{T}}+ C\|u- u_{\mathcal{T}}\|_{1, \mathcal{T}}.
\end{equation}
Combining (\ref{nonconf_part_mid2}) with Lemma \ref{interpolation_error} and Theorems \ref{thm:WilsonConvergence} and \ref{Wilson eigenvaluepositive}, we derive that
\begin{equation}\label{nonconf_part_mid3}
|Q_{\mathcal{T}}u-u_{\mathcal{T},c}|_1 \leq Ch |u|_2.
\end{equation}
%Lemma \ref{interpolation_error} and (\ref{nonconf_part_mid3}) imply that
%\begin{equation}\label{nonconf_part_mid4}
%|u- u_{\mathcal {T},c}|_{1} \leq |u- Q_{\mathcal{T}}u|_{1}+|Q_{\mathcal{T}}u- u_{\mathcal {T},c}|_{1} \leq C h|u|_2.
%\end{equation}
Substituting (\ref{nonconf_part_mid6}) and (\ref{nonconf_part_mid3}) into (\ref{nonconf_part_mid1}), we derive the first desired inequality.

We next verify the second inequality of this lemma. By the variable transformation, we derive that
$$
\|u_{\mathcal {T},n}\|_{0,\Omega}^2 = \sum\limits_{K\in\mathcal{T}} \int_{K} |u_{\mathcal {T},n}|^2 dx = \sum\limits_{K\in\mathcal{T}} \frac{|K|}{4} \int_{\hat{K}} |\hat{u}_{\mathcal {T},n}|^2 d\hat{x}.
$$
and
$$
\|u_{\mathcal {T},n}\|_{1,\mathcal {T}}^2 = \sum\limits_{K\in\mathcal{T}} \int_{K} |\nabla u_{\mathcal {T},n}|^2 dx = 4\sum\limits_{K\in\mathcal{T}} \int_{\hat{K}} (\nabla \hat{u}_{\mathcal {T},n})^T \mathbf{F}_K \nabla \hat{u}_{\mathcal {T},n} d\hat{x}.
$$
Since $\mathscr{T}$ is regular, $\|u_{\mathcal {T},n}\|_{1,\mathcal {T}}^2$ is equivalent to $\sum\limits_{K\in\mathcal{T}} \int_{\hat{K}} |\nabla \hat{u}_{\mathcal {T},n}|^2 d\hat{x}$. By directly calculation, we easily obtain that
$$
\int_{\hat{K}} |\hat{u}_{\mathcal {T},n}|^2 d\hat{x} \leq C \int_{\hat{K}} |\nabla \hat{u}_{\mathcal {T},n}|^2 d\hat{x}.
$$
Thus, we derive that
\begin{equation}\label{nonconf_part_mid5}
\|u_{\mathcal {T},n}\|_{0,\Omega}^2 \leq  Ch^2 \|u_{\mathcal {T},n}\|_{1,\mathcal {T}}^2.
\end{equation}
The first inequality of this lemma combining with (\ref{nonconf_part_mid5}) immediately yields the second desired inequality.
\end{proof}

Let $u\in \mathbb{H}_0^1(\Omega)\cap \mathbb{H}^2(\Omega)$ be the solution of
(\ref{eq:poisson_equation}). According to the decomposition (\ref{decompositionFVMsol}) and Lemma \ref{nonconf_part}, we easily obtain
\begin{equation}\label{L2_mid1}
\|u-u_{\mathcal{T}}\|_0 \leq \|u-u_{\mathcal{T},c}\|_0 + \|u_{\mathcal{T},n}\|_0 \leq \|u-u_{\mathcal{T},c}\|_0 + Ch^2 |u|_2.
\end{equation}
In the following, we devote ourselves to estimating $\|u-u_{\mathcal{T},c}\|_0$. To this end,
we introduce an auxiliary problem: find $\varphi\in \mathbb{H}^2(\Omega)$ such that
\begin{equation}\label{dual_problem}
-\Delta \varphi =u-u_{\mathcal{T},c} \quad \mbox{in}\  \Omega   \quad \mbox{and} \quad \varphi=0 \quad \mbox{on}\  \partial \Omega.
\end{equation}
It is well-known that (cf. \cite{G})
\begin{equation}\label{dual_problem_mid1}
\|\varphi\|_{2} \leq C\|u-u_{\mathcal{T},c}\|_0.
\end{equation}
For $w,v\in \mathbb{H}^1(\Omega)$, we define the bilinear form
$$
e_K(w,v):=\int_{K} \nabla w \cdot \nabla v    \quad \mbox{and} \quad  a(w,v):= \sum\limits_{K\in\mathcal{T}} e_K(w,v).
$$
%where $E_K(w,v):=\int_{K} \nabla w \cdot \nabla v$.

\begin{lemma}\label{L2error_decom}
It holds that
\begin{equation}\label{L2error_decom_mid0}
\|u-u_{\mathcal {T},c}\|_{0}^2 = a(u-u_{\mathcal {T},c}, \varphi -Q_{\mathcal{T}}\varphi ) +  a(u-u_{\mathcal{T}}, Q_{\mathcal{T}}\varphi) + a(u_{\mathcal {T},n}, Q_{\mathcal{T}}\varphi- \varphi) + a(u_{\mathcal {T},n}, \varphi).
\end{equation}
\end{lemma}
\begin{proof}
An application of the Green's formula to (\ref{dual_problem}), we get that
\begin{equation}\label{L2error_decom_mid1}
\|u-u_{\mathcal {T},c}\|_{0}^2 = a(u-u_{\mathcal {T},c}, \varphi).
\end{equation}
Obviously,
\begin{eqnarray*}
 %\nonumber to remove numbering (before each equation)
   a(u-u_{\mathcal {T},c}, \varphi) &=& a(u-u_{\mathcal {T},c}, \varphi -Q_{\mathcal{T}}\varphi ) +  a(u-u_{\mathcal {T},c}, Q_{\mathcal{T}} \varphi)\\
   &=& a(u-u_{\mathcal {T},c}, \varphi -Q_{\mathcal{T}}\varphi ) +  a(u-u_{\mathcal {T}}, Q_{\mathcal{T}} \varphi) +  a(u_{\mathcal {T},n}, Q_{\mathcal{T}} \varphi)\\
   &=& a(u-u_{\mathcal {T},c}, \varphi -Q_{\mathcal{T}}\varphi ) +  a(u-u_{\mathcal {T}}, Q_{\mathcal{T}} \varphi) + a(u_{\mathcal {T},n}, Q_{\mathcal{T}}\varphi- \varphi) + a(u_{\mathcal {T},n}, \varphi).
\end{eqnarray*}
%$$
%a(u-u_{\mathcal {T},c}, \varphi)=a(u-u_{\mathcal {T},c}, \varphi -Q_{\mathcal{T}}\varphi ) +  a(u-u_{\mathcal {T},c}, Q_{\mathcal{T}} \varphi)
%=a(u-u_{\mathcal {T},c}, \varphi -Q_{\mathcal{T}}\varphi ) +  a(u-u_{\mathcal {T}}, Q_{\mathcal{T}} \varphi) +  a(u_{\mathcal {T},n}, Q_{\mathcal{T}} \varphi).
%$$
Thus, the desired result of this lemma is proved.
\end{proof}

We next estimate the terms on the right-hand side of (\ref{L2error_decom_mid0}) respectively. The following lemma gives the estimation of the first term.
\begin{lemma}\label{L2error1}
If $\mathscr{T}$ is regular, then there holds
$$
|a(u-u_{\mathcal {T},c}, \varphi -Q_{\mathcal{T}}\varphi )|\leq C h^2 |u|_2 \|u-u_{\mathcal {T},c}\|_0.
$$
\end{lemma}
\begin{proof}
Using Lemma \ref{interpolation_error} and (\ref{dual_problem_mid1}), we obtain that
\begin{equation}\label{L2error1_mid2}
|a(u-u_{\mathcal {T},c}, \varphi -Q_{\mathcal{T}}\varphi )|\leq |u-u_{\mathcal {T},c}|_1 \cdot |\varphi -Q_{\mathcal{T}}\varphi|_1  \leq C h|u-u_{\mathcal {T},c}|_1 \cdot \|u-u_{\mathcal {T},c}\|_0.
\end{equation}
From Theorems \ref{thm:WilsonConvergence} and \ref{Wilson eigenvaluepositive} and Lemma \ref{nonconf_part}, we derive
\begin{equation}\label{L2error1_mid2}
|u-u_{\mathcal {T},c}|_1 \leq \|u-u_{\mathcal{T}}\|_{1,\mathcal{T}} + \|u_{\mathcal{T},n}\|_{1,\mathcal{T}} \leq C h |u|_2.
\end{equation}
Substituting (\ref{L2error1_mid2}) into (\ref{L2error1_mid2}) completes the proof of this lemma.
\end{proof}

The results of the next lemma can be found in \cite{LL2}.
\begin{lemma}\label{inverse_inequaltiy}
For any $K\in \mathcal {T}$ and any function $w\in \mathbb{H}^3(K)$
$$
|\hat{w}|_{m,\hat{K}}\leq Ch^{m-1}|u|_{m,K}, \ m=0,1, \quad |\hat{w}|_{2,\hat{K}}\leq Ch(|u|_{1,K}+|u|_{2,K}), \quad |\hat{w}|_{3,\hat{K}}\leq Ch^{2}\|u\|_{3,K}.
$$
%$$
%|\hat{w}_{2,\hat{K}}|\leq Ch(|u|_{1,K}+|u|_{2,K}),
%$$
%and
%$$
%|\hat{w}_{3,\hat{K}}|\leq Ch^{2}\|u\|_{3,K}
%$$
\end{lemma}

We introduce some notations. For each $K\in \mathcal {T}$, set
$$
w_{ij}=(Q_{\mathcal{T}}\varphi)_{i,K}-(Q_{\mathcal{T}}\varphi)_{j,K}, \  i,j\in \mathbb{N}_{4} \quad  \mbox{and} \quad w_{1234}=(Q_{\mathcal{T}}\varphi)_{1,K}-(Q_{\mathcal{T}}\varphi)_{2,K}+(Q_{\mathcal{T}}\varphi)_{3,K}-(Q_{\mathcal{T}}\varphi)_{4,K}.
$$
where $(Q_{\mathcal{T}}\varphi)_{i,K}, i\in \mathbb{N}_{4}$ are defined as in (\ref{coef_on_K}).
From (3.12) and (3.13) of \cite{LL2}, we get that
\begin{equation}\label{wij}
|w_{ij}| \leq C |Q_{\mathcal{T}}\varphi|_{1,K},\  i,j \in \mathbb{N}_{4} \quad  \mbox{and} \quad |w_{1234}| \leq C h \|Q_{\mathcal{T}}\varphi\|_{2,K}.
\end{equation}
Let $\hat{x}:=(\hat{x}_1, \hat{x}_2)^T\in \hat{K}$. Define
$$
F_1(\hat{x}):= (w_{12}-w_{34},0)\mathbf{F}_K \nabla (\hat{u}-\hat{u}_{\mathcal{T}}),\  R_1(F_1(\hat{x})):=\frac{1}{2} \left(\int_{-1}^0 (\hat{x}_1+1)^2\frac{\partial^2 F_1}{\partial \hat{x}_1^2} d\hat{x}_1 +  \int_{0}^1 (\hat{x}_1-1)^2\frac{\partial^2 F_1}{\partial \hat{x}_1^2} d\hat{x}_1 \right)
$$
and
$$
F_2(\hat{x}):= (0, w_{13}-w_{24})\mathbf{F}_K \nabla (\hat{u}-\hat{u}_{\mathcal{T}}),\ R_2(F_2(\hat{x})):=\frac{1}{2} \left(\int_{-1}^0 (\hat{x}_2+1)^2\frac{\partial^2 F_2}{\partial \hat{x}_2^2} d\hat{x}_2 +  \int_{0}^1 (\hat{x}_2-1)^2\frac{\partial^2 F_2}{\partial \hat{x}_2^2} d\hat{x}_2 \right)
$$

We are ready to estimate the second term on the right-hand side of (\ref{L2error_decom_mid0}) in the next lemma.
\begin{lemma}\label{L2error2}
If $\mathscr{T}$ is regular and $u\in \mathbb{H}_0^1(\Omega)\cap \mathbb{H}^3(\Omega)$, then there holds
$$
|a(u-u_{\mathcal {T}}, Q_{\mathcal{T}} \varphi)| \leq ch^2 \|u\|_3 \|u-u_{\mathcal{T},c}\|_0.
$$
\end{lemma}
\begin{proof}
Let $Q^*_{\mathcal{T}}\varphi :=\Pi_{\mathcal{T}^{*}}(Q_{\mathcal{T}} \varphi)$. Noting that $Q^*_{\mathcal{T}}\varphi \in \mathbb{H}^1_{\mathcal{T}^{*}}(\Omega)$ and
$a_{\mathcal{T}}(u-u_{\mathcal{T}}, Q^*_{\mathcal{T}}\varphi)=0$, we get
\begin{equation}\label{L2error2_mid1}
a(u-u_{\mathcal {T}}, Q_{\mathcal{T}} \varphi) = a(u-u_{\mathcal {T}}, Q_{\mathcal{T}} \varphi) - a_{\mathcal{T}}(u-u_{\mathcal{T}}, Q^*_{\mathcal{T}}\varphi)= \sum_{K\in\mathcal{T}} \left(e_K (u-u_{\mathcal {T}}, Q_{\mathcal{T}} \varphi) - a_K (u-u_{\mathcal {T}}, Q_{\mathcal{T}}^* \varphi) \right).
\end{equation}
%We begin to estimate $e_K(\cdot,\cdot)-a_K(\cdot,\cdot)$.
For $K=\Theta\{P_1, P_2, P_3, P_4\}\in \mathcal{T}$, we use $M_i, i\in \mathbb{N}_4$ to denote the middle point of the edge $P_iP_j$ with $P_5:=P_1$ and use $Q$ to denote its center.
Similar arguments as those in Theorem 1 of \cite{LL2} reveal that
\begin{equation} \label{L2error2_mid2}
\begin{split}
e_K (u-u_{\mathcal {T}}, Q_{\mathcal{T}} \varphi)&- a_K (u-u_{\mathcal {T}}, Q_{\mathcal{T}}^* \varphi) =  \int_{-1}^1 R_1(F_1(\hat{x})) d\hat{x}_2 + \int_{-1}^1 R_2(F_2(\hat{x})) d\hat{x}_1 \\
 & +  w_{1234} \int_{\hat{K}}(\hat{x}_2, \hat{x}_1)\mathbf{F}_K \nabla (\hat{u}-\hat{u}_{\mathcal{T}}) d\hat{x}
 + w_{1234}\int_{\widehat{M_3QM_2}}\nabla (u-u_{\mathcal{T}})\cdot \mathbf{n} ds,
 \end{split}
 \end{equation}
where $\widehat{M_3QM_2}:=\overline{M_3 Q}\cup \overline{QM_2}$ and $\mathbf{n}$ is the outward unit normal vector on $\widehat{M_3QM_2}$. We begin to estimate the terms on the right-hand side of (\ref{L2error2_mid2})

Obviously
\begin{equation}\label{L2error2_mid3}
\left|\int_{-1}^1 R_1(F_1(\hat{x})) d\hat{x}_2\right| \leq  \int_{\hat{K}} (\hat{x}_1^2+1) \left|\frac{\partial^2 F_1}{\partial \hat{x}_1^2}\right| d\hat{x} \leq C  \|\frac{\partial^2 F_1}{\partial \hat{x}_1^2}\|_{0,\hat{K}}.
\end{equation}
The regularity of $\mathscr{T}$ and (\ref{wij}) implies that
\begin{equation}\label{L2error2_mid4}
\left|\frac{\partial^2 F_1}{\partial \hat{x}_1^2}\right| = \left|\frac{h_{2,K}^2}{|K|}(w_{12}-w_{34})\frac{\partial^3(\hat{u}-\hat{u}_{\mathcal{T}})}{\partial \hat{x}_1^3}\right| \leq C |Q_{\mathcal{T}}\varphi|_{1,K} \cdot \left|\frac{\partial^3(\hat{u}-\hat{u}_{\mathcal{T}})}{\partial \hat{x}_1^3}\right|.
\end{equation}
Combining (\ref{L2error2_mid3}), (\ref{L2error2_mid4}) and Lemma \ref{inverse_inequaltiy} yields that
\begin{equation}\label{L2error2_mid9}
\left|\int_{-1}^1 R_1(F_1(\hat{x})) d\hat{x}_2\right| \leq Ch^2 \|u\|_{3,K} |Q_{\mathcal{T}}\varphi|_{1,K}.
\end{equation}
In the same way as above, we have that
\begin{equation}\label{L2error2_mid10}
\left|\int_{-1}^1 R_2(F_2(\hat{x})) d\hat{x}_1\right| \leq Ch^2 \|u\|_{3,K} |Q_{\mathcal{T}}\varphi|_{1,K}.
\end{equation}
By making use the regularity of $\mathscr{T}$ and the Cauchy-Schwartz inequality, we derive that
$$
\left|w_{1234} \int_{\hat{K}}(\hat{x}_2, \hat{x}_1)\mathbf{F}_K \nabla (\hat{u}-\hat{u}_{\mathcal{T}}) d\hat{x}\right| \leq C|w_{1234}| \cdot |\hat{u}-\hat{u}_{\mathcal{T}}|_{1,\hat{K}}.
$$
Thus, applying Lemma \ref{inverse_inequaltiy} and (\ref{wij}) to the above inequality leads to that
\begin{equation}\label{L2error2_mid5}
\left|w_{1234} \int_{\hat{K}}(\hat{x}_2, \hat{x}_1)\mathbf{F}_K \nabla (\hat{u}-\hat{u}_{\mathcal{T}}) d\hat{x}\right| \leq Ch \|Q_{\mathcal{T}}\varphi\|_{2,K} |u-u_{\mathcal{T}}|_{1,K}.
\end{equation}
Using the regularity of $\mathscr{T}$, the variable transformation from $K$ to $\hat{K}$ and the trace theorem, we have that
\begin{equation}\label{L2error2_mid6}
\int_{\widehat{M_3QM_2}}\nabla (u-u_{\mathcal{T}})\cdot \mathbf{n} ds \leq C \|\nabla (\hat{u}-\hat{u}_{\mathcal{T}})\|_{1,\hat{K}}.
\end{equation}
By Lemma \ref{inverse_inequaltiy}, we get that
\begin{equation}\label{L2error2_mid7}
\|\nabla (\hat{u}-\hat{u}_{\mathcal{T}})\|_{1,\hat{K}} \leq C (|u-u_{\mathcal{T}}|_{1,K}+ h|u-u_{\mathcal{T}}|_{2,K})
\end{equation}
From (\ref{L2error2_mid6}), (\ref{L2error2_mid7}) and (\ref{wij}), we get that
\begin{equation}\label{L2error2_mid8}
\left|w_{1234}\int_{\widehat{M_3QM_2}}\nabla (u-u_{\mathcal{T}})\cdot n ds\right|\leq C h(|u-u_{\mathcal{T}}|_{1,K}+ h|u-u_{\mathcal{T}}|_{2,K})\|Q_{\mathcal{T}}\varphi\|_{2,K}.
\end{equation}

Finally, combining (\ref{L2error2_mid2}) with (\ref{L2error2_mid9}), (\ref{L2error2_mid10}), (\ref{L2error2_mid5}) and (\ref{L2error2_mid8}), we obtain
$$
\left|e_K (u-u_{\mathcal {T}}, Q_{\mathcal{T}} \varphi)- a_K (u-u_{\mathcal {T}}, Q_{\mathcal{T}}^* \varphi)\right| \leq C h^2 \|u\|_{3,K}|Q_{\mathcal{T}}\varphi|_{1,K} + Ch(|u-u_{\mathcal{T}}|_{1,K}+ h|u-u_{\mathcal{T}}|_{2,K})\|Q_{\mathcal{T}}\varphi\|_{2,K}.
$$
This combined with (\ref{L2error2_mid1}), Theorems \ref{thm:WilsonConvergence} and \ref{Wilson eigenvaluepositive} and (\ref{dual_problem_mid1}) leads to the desired result of this lemma.
\end{proof}

The third term on the right-hand side of (\ref{L2error_decom_mid0}) is estimated in the next lemma.
\begin{lemma}\label{L2error3}
If $\mathscr{T}$ is regular,
 then there holds
$$
\left|a(u_{\mathcal {T},n}, Q_{\mathcal{T}}\varphi- \varphi) \right| \leq Ch^2 |u|_2 \|u-u_{\mathcal{T},c}\|_0.
$$
\end{lemma}
\begin{proof}
Using Lemma \ref{nonconf_part} and Lemma \ref{interpolation_error}, we have that
$$
\left|a(u_{\mathcal {T},n}, Q_{\mathcal{T}}\varphi- \varphi)\right| \leq C h^2 |u|_2 |\varphi|_2.
$$
This combined with (\ref{dual_problem_mid1}) yields the desired result of this lemma.
\end{proof}

In the next lemma, we present the estimation of the last term on the right-hand side of (\ref{L2error_decom_mid0}).
\begin{lemma}\label{L2error4}
If $\mathscr{T}$ is regular,
 then there holds
$$
\left|a(u_{\mathcal {T},n}, \varphi)\right| \leq Ch^2 |u|_2 \|u-u_{\mathcal{T},c}\|_0.
$$
\end{lemma}
\begin{proof}
By the Green's formula, we get that
\begin{equation}\label{L2error4_mid1}
a(u_{\mathcal {T},n}, \varphi) = \sum_{K\in\mathcal{T}} \int_{\partial K} u_{\mathcal {T},n} \nabla \varphi \cdot \mathbf{n}  - \sum_{K\in\mathcal{T}} \int_{K} u_{\mathcal {T},n} \Delta \varphi
\end{equation}
%The first term on the right-hand side of (\ref{L2error4_mid1}) can be treated by the similar technique as that used in Theorem 5 of \cite{Sh}.
Applying the similar technique as that used in Theorem 5 of \cite{Sh}, we derive that (\ref{L2error4_mid1})
\begin{equation}\label{L2error4_mid4}
\left|\sum_{K\in\mathcal{T}} \int_{\partial K} u_{\mathcal {T},n} \nabla \varphi \cdot \mathbf{n} \right|\leq C h \|\varphi\|_2 \|u_{\mathcal {T},n}\|_{1,\mathcal {T}}.
\end{equation}
%This combined with (\ref{dual_problem_mid1}) and Lemma \ref{nonconf_part} leads to that
%\begin{equation}\label{L2error4_mid2}
%\left|\sum_{K\in\mathcal{T}} \int_{\partial K} u_{\mathcal {T},n} \nabla \varphi \cdot n \right|\leq C h^2 |u|_2 \|u-u_{\mathcal{T},c}\|_0.
%\end{equation}
An application of the Cauchy-Schwartz inequality implies that
\begin{equation}\label{L2error4_mid3}
\left|\sum_{K\in\mathcal{T}} \int_{K} u_{\mathcal {T},n} \Delta \varphi \right|\leq |\varphi|_2 \|u_{\mathcal{T},n}\|_0.
\end{equation}
Then, from (\ref{L2error4_mid1})-(\ref{L2error4_mid3}),  (\ref{dual_problem_mid1}) and Lemma \ref{nonconf_part}, we get the desired result of this lemma.
\end{proof}

From Lemma \ref{L2error_decom}, Lemma \ref{L2error1}, Lemma \ref{L2error2}, Lemma \ref{L2error3} and Lemma \ref{L2error4}, we can obtain the following $L^2$ error estimate for the hybrid Wilson FVM.
\begin{theorem}\label{Thm_L2error}
Let $u\in \mathbb{H}_0^1(\Omega)\cap \mathbb{H}^3(\Omega)$ be the solution of
(\ref{eq:poisson_equation}) and $u_{\mathcal{T}}\in \mathbb{U}_\mathcal{T}$ be the solution of the hybrid Wilson FVM.
If $\mathscr{T}$ is regular, then there holds
$$
\|u-u_\mathcal{T}\|_{0}\le C h^2 \|u\|_{3}.
$$
\end{theorem}

%%%%%%%%%%%%%%%%%%%%%%%%%%%%%%%%%%%%%%%%%%%%%%%%%%%%%%%%%%%%%%%%%%%%%%%
\section{Numerical Examples}
%%%%%%%%%%%%%%%%%%%%%%%%%%%%%%%%%%%%%%%%%%%%%%%%%%%%%%%%%%%%%%%%%%%%%%%

In this section, we present the numerical results of the C-R FVM to confirm the theoretical analysis in this paper.
The experiments here are performed on a personal computer with 2.30 GHz CPU and 4 Gb RAM. Moreover,
Matlab 7.7 is used as the testing platform and the direct algorithm is used to solve the resulting linear systems.

We consider solving the Poisson equation \eqref{eq:poisson_equation} with $f(x,y):= 2(x^2+y^2-x-y)$ and $\Omega :=(0,1)\times (0,1)$.
The exact solution of the boundary value problem is given by
$
u(x,y) = -x(x-1)y(y-1), \ (x,y)\in [0,1]\times [0,1].
$
From \cite{Z}, we know that the regular condition (\ref{regularity}) of the family $\mathscr{T}$ of the triangulations is equivalent to that there exists a positive constant $\theta_{\inf}$ such that
$$
        \theta_{\min,K} \geq \theta_{\inf}, \quad \textrm{for all } K \in
\bigcup_{\mathcal{T}\in \mathscr{T}} \mathcal{T},
$$
where $\theta_{\min,K}$ denotes the minimum angle of the triangle $K$.
We fist subdivide the region $[0,1]\times [0,1]$ to $M\times N$ rectangles with equal size. Then the triangle mesh of ${\Omega}$ is obtained by connecting the diagonal lines of the resulting rectangles. The triangulation of the case $M=2$ and $N=4$ is illustrated by Figure \ref{fig:partition}. Without loss of generality, we may assume that $M\leq N$.
%We use $\theta_{\min}$ to denote the minimum angle of all the triangles in the triangulation.
Obviously,
$$
\tan \theta_{\min,K} = M/N.
$$
We may adjust $M$ and $N$ so as to obtain different triangulations with different minimum angles.

\begin{figure}[ht!]
\centering
\includegraphics[width=0.30\textwidth]{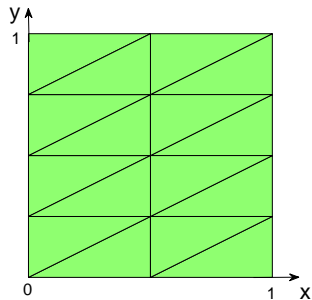}
\caption{A triangulation of the region $\bar{\Omega}$}\label{fig:partition}
\end{figure}

We list the $\|\cdot\|_{1,\mathcal {T}}$-errors and the convergence orders (C.O.) for the C-R FVM under different triangulations with different minimum angles in Table \ref{HLCresult}, where $n$ is the number of unknowns of the resulting linear system.
It follows from  Theorem \ref{thm:CRconvergence} that when $\theta_{\min}=45^{\circ}$, $\theta_{\min}\approx18.43^{\circ}$ or $\theta_{\min}\approx2.86^{\circ}$, the convergence order of the $\|\cdot\|_{1,\mathcal {T}}$-error between the exact solution $u$ of the Poisson equation and the
solution $u_\mathcal {T}$ of the C-R FVM is $O(h)$, which is validated in the numerical results in Table \ref{HLCresult}.

\begin{table}[h]
\caption{The numerical results of the C-R FVM}\label{HLCresult} \vskip 4mm
\centerline{\begin{tabular}{cccc|cccc|cccc} \hline \hline
 $\theta_{\min}$ &  $=$ &  $45^{\circ}$ &                   & $\theta_{\min}$ &  $\approx$ &  $18.43^{\circ}$ &           & $\theta_{\min}$ &  $\approx$ &  $2.86^{\circ}$  \\ \hline
$(M, N)$ & $n$  & $\|\cdot\|_{1,\mathcal {T}}$ & $\mbox{C.O.}$    & $(M, N)$ & $n$  & $\|\cdot\|_{1,\mathcal {T}}$ & $\mbox{C.O.}$    & $(M, N)$ & $n$  & $\|\cdot\|_{1,\mathcal {T}}$ & $\mbox{C.O.}$\\ \hline
$(2, 2)$ & $16$  & 8.42e-2   &                            &    $(1, 3)$ & $13$  & 9.44e-2 &                           &   $(1, 20)$  &   $81$
&  8.41e-2       &    \\
$(4, 4)$ & $56$ & 4.21e-2   & 0.99                     &    $(2, 6)$ & $44$  &6.44e-2   & 0.55                    &     $(2, 40)$ & $282$
  &6.16e-2 &  0.45\\
$(8, 8)$ & $208$ & 2.10e-2   & 1.00                       &    $(4, 12)$ & $160$  & 3.17e-2   & 1.02                 &     $(4, 80)$ & $1044$  &3.02e-2    & 1.03\\
$(16, 16)$ & $800$  & 1.05e-2  & 1.00                    &    $(8, 24)$ & $608$  & 1.57e-2  & 1.01                 &     $(8, 160)$ & $4008$
& 1.49e-2   & 1.02\\
$(32, 32)$ & $3136$  & 5.25e-3  & 1.00                    &    $(16, 48)$ & $2368$  & 7.84e-3  & 1.00                 &    $(16, 320)$ & $15696$  &7.45e-3  & 1.00\\
$(64, 64)$ & $12416$  & 2.63e-3  & 1.00                   &    $(32, 96)$ & $9344$  & 3.92e-3  & 1.00                &     $(32,640)$ & $62112$  & 3.72e-3  & 1.00   \\
$(64, 64)$ & $49408$  & 1.31e-3  & 1.00                   &    $(32, 96)$ & $37120$  & 1.96e-3  & 1.00                &     $(64,1280)$ & $247104$  & 1.86e-3  & 1.00  \\ \hline
\end{tabular}}
\end{table}

\vspace{5mm}

\noindent \emph{Acknowledgment. } The first author wishes to thank Professor Chunjia Bi for useful discussions.

\end{document}